%% file: revised.tex
\documentclass[a4paper,11pt]{article}

\usepackage[utf8]{inputenc}
\DeclareUnicodeCharacter{00A0}{~}

\usepackage[english]{babel}

\usepackage{a4wide}
\usepackage{amssymb,amsmath,amscd,amsfonts,amsthm,bbm,mathrsfs,enumerate}
\usepackage[colorlinks=true, linkcolor=blue, citecolor=blue]{hyperref}
\usepackage{graphicx}
\usepackage{tikz,pgfplots}
\usepackage{color}
\usepackage{csquotes}
\usepackage{authblk}
\usepackage{shortcuts}
\usepackage{comment}
\usepackage[title]{appendix}

\newcommand{\bsd}{\boldsymbol d}

\newtheorem{lemma}{Lemma}
\newtheorem{proposition}{Proposition}
\newtheorem{definition}{Definition}
\newtheorem{theorem}{Theorem}
\newtheorem{example}{Example}
\newtheorem{remark}{Remark}
\newtheorem{assumption}{Assumption}
\newtheorem{corollary}{Corollary}

\newcounter{compteur-enumeration}

\renewcommand{\cT}{T}
\renewcommand{\cM}{M}
\providecommand{\keywords}[1]{\textbf{\textbf{Keywords. }} #1}

\title{Stochastic Subgradient Descent Escapes Active Strict Saddles on Weakly Convex Functions}
\author{Pascal Bianchi, Walid Hachem, Sholom Schechtman}
\begin{document}
\maketitle

\begin{abstract}
In non-smooth stochastic optimization, we establish the non-convergence of the
stochastic subgradient descent (SGD) to the critical points recently called
active strict saddles by Davis and Drusvyatskiy.  Such points lie on a manifold
$M$ where the function $f$ has a direction of second-order negative curvature.
Off this manifold, the norm of the Clarke subdifferential of $f$ is
lower-bounded.  We require two conditions on $f$. The first assumption is a
Verdier stratification condition, which is a refinement of the popular Whitney
stratification.  It allows us to establish a strengthened version of the
projection formula of Bolte \emph{et al.} for Whitney stratifiable functions,
and which is of independent interest.  The second assumption, termed the angle
condition, allows to control the distance of the iterates to $M$.  When $f$ is
weakly convex, our assumptions are generic.  Consequently, generically in the
class of definable weakly convex functions, SGD converges to a local
minimizer.
\end{abstract}

\keywords{Non-smooth optimization, stochastic gradient descent, avoidance of traps, Clarke subdifferential, stratification, weak convexity}
\section{Introduction}

Stochastic approximation algorithms that operate on non-convex and non-smooth
functions have recently attracted a great deal of attention, owing to their
numerous applications in machine learning and in high-dimensional statistics.
The archetype of such algorithms is the so-called Stochastic Subgradient
Descent (SGD), which reads as follows. Given a locally Lipschitz function $f :
\bbR^d \rightarrow \bbR$ which is not necessarily smooth nor
convex, the $\bbR^d$--valued sequence $(x_n)$ of iterates generated by such an
algorithm satisfy the inclusion
\begin{equation}
\label{eq:sgd-intro}
x_{n+1} \in x_n - \gamma_n \partial f(x_n) + \gamma_n \eta_{n+1},
\end{equation}
where the set--valued function $\partial f$ is the so-called Clarke
subdifferential of $f$, the sequence $(\gamma_n)$ is a sequence of positive
step sizes converging to zero, and $\eta_{n+1}$ is a zero-mean random vector on
$\bbR^d$ which presence is typically due to the partial knowledge of $\partial
f$ by the designer.  It is desired that $(x_n)$ converges to the set of local
minimizers of the function $f$.

Before delving into the subject of convergence toward minimizers, let us
first consider the set $\cZ:=\{ x\in \bbR^d: 0 \in \partial f(x)\}$ of
\emph{Clarke critical points} of $f$, which is generally larger than the set of
minimizers, in the non-convex case. In order to ensure the convergence of
$(x_n)$ to $\cZ$, the sole local Lipschitz property of $f$ is not enough
(see~\cite{dan_dru19} for a counterexample), and some form of structure for the
function $f$ is required.  Since the work of Bolte \emph{et al.}
\cite{bolte2007clarke} in optimization theory, it is well known that the
so-called \emph{definable on an o-minimal structure} (henceforth definable)
functions, which belong to the family of \emph{Whitney stratifiable} functions
(see Section~\ref{sec-prel} below), is relevant for the convergence analysis of
$(x_n)$ and beyond.  This class of functions is general enough so as to contain
all the functions that are practically used in machine learning, statistics, or
applied optimization.  In this framework, the almost sure convergence of
$(x_n)$ to $\cZ$ was established by Davis \emph{et al.}~ in
\cite{dav-dru-kak-lee-19} (in the sense that $\limsup \dist(x_n, \cZ) = 0$, the convergence of $(x_n)$ to a unique element of $\cZ$ is, for now, only a conjecture). Another work in the same line is
\cite{maj-mia-mou18}. Bolte and Pauwels \cite{bolte2019conservative} generalize
the algorithm \eqref{eq:sgd-intro} by replacing $\partial f$ with an arbitrary
so-called conservative field. The constant step size regime $\gamma_n \equiv
\gamma$ is considered in \cite{bia-hac-sch-SVA}.

Thanks to these contributions, we know that any convergent subsequence of $(x_n)$ will have a limit in $\cZ$. However, as said above, $\cZ$ is in general strictly larger
than the set of minimizers, and can contain ``spurious'' points such as local
maximizers or saddle points.  The issue of the \emph{non-convergence} of the
sequence given by \eqref{eq:sgd-intro} toward spurious critical points is
therefore crucial. The present paper investigates this issue.

Before getting into the core of our subject, it is useful to make a quick
overview of the results devoted to the avoidance of spurious critical points by
the iterative algorithms. The rich literature on this subject has been almost entirely
devoted to the smooth setting.  In this framework, the research has followed
two main axes:

\begin{itemize}

\item The noisy case, where the analogue of the sequence $(\eta_n)$ in the
smooth version of Algorithm \eqref{eq:sgd-intro} is non zero.  Here, the
seminal works of Pemantle~\cite{pem-90} and Brandière and
Duflo~\cite{bra-duf-96} allow to establish the non-convergence of the
Stochastic Gradient Descent (and, more generally, of Robbins-Monro algorithms)
to a certain type of spurious critical points, sometimes referred to as
\emph{traps} or \emph{strict saddle}. A critical point of a smooth function $f$
is called a trap if the Hessian matrix of $f$ at this point admits at least one
negative eigenvalue.  With probability one, the sequence $(x_n)$ cannot
converge to a trap, provided that the projection of the random perturbation
$\eta_{n}$ onto the eigenspace corresponding to the negative eigenvalues of
the Hessian matrix (henceforth, eigenspace of negative curvature) has a non
vanishing variance.

\item The noiseless case where $\eta_n \equiv 0$, studied for smooth functions
by \cite{lee16_avoid}. Here the authors show that for Lebesgue-almost all initialization points, the algorithm with
constant step will avoid the traps.
\end{itemize}
While both of these approaches rely on the center-stable invariant manifold
theorem which finds its roots in the work of Poincaré, they are different in
spirit. Indeed, in \cite{lee16_avoid} the trap avoidance is due to the random
initialization of the algorithm, whereas in \cite{bra-duf-96, pem-90}, it is
due to the inherent stochasticity brought by the sequence $(\eta_n)$.

We now get back to the non-smooth case. Here, the only paper that tackles the
problem of the spurious points avoidance is, up to our knowledge, the recent
contribution \cite{dav_dru21} of Davis and Drusvyatskiy.
The spurious points that were considered in this reference are the so-called
\emph{active strict saddles}.  Formally, a critical point is an active strict
saddle if it lies on a manifold $\cM$ such that {\sl i)} $f$ varies sharply
outside $\cM$, {\sl ii)} the restriction of $f$ to $\cM$ is smooth, and {\sl
iii)} the Riemannian Hessian of $f$ on $\cM$ has at least one negative
eigenvalue. For instance, the function $f : \bbR^2\to \bbR, (y,z)\mapsto | z|
- y^2$ admits the point $(0,0)$ as an active strict saddle with $\cM =  \bbR \times \{ 0\}$, and the restriction of $f$ to $\cM$ is the function $f_\cM(y, 0)
= - y^2$, which  has a second-order negative curvature. In this setting, and assuming
that $f$ is weakly convex, the
article~\cite{dav_dru21} focuses on the noiseless case, and study variants of the (implicit)
\emph{proximal point algorithm} rather than the (explicit) subgradient descent.
Similarly to \cite{lee16_avoid}, they show that for Lebesgue-almost every initialization point,
different versions of the proximal algorithm avoid active strict saddles with probability one. Such a result is possible due to the fact that proximal methods implicitly run a gradient descent on a regularized version of $f$ - the Moreau envelope (which is well-defined due to the weak-convexity of $f$).

Contrary to \cite{dav_dru21}, the algorithm~\eqref{eq:sgd-intro} studied in this paper is explicit,
meaning that it does not require the computation of a proximal operator associated with the non-smooth function.
In this situation, the sole randomization of the initial point is not sufficient to expect an avoidance of active strict saddles.
Here, in the same line as \cite{pem-90,bra-duf-96}, our analysis strongly relies on the presence of
the additive random perturbation~$\eta_n$.

In the framework of weakly convex definable functions, we investigate the problem of the
avoidance of the active strict saddle points.
Our approach goes as
follows. First, we need to show that the iterates $(x_n)$ converge sufficiently
fast to $\cM$, thanks to the sharpness of $f$ outside this manifold. To that
end, our first tool is an assumption that we term as the \emph{angle condition}.
Roughly, this assumption provides a lower bound on the inner product between
the subgradients of $f$ at $x$ and the normal direction from $\cM$ to $x$ when
the point $x$ is near $\cM$. The angle
condition allows to control the distance between the iterate $x_n$ of
Algorithm~\eqref{eq:sgd-intro} and the manifold $\cM$. Second,
we  rely on the fact that when $f$ is definable, its graph always
admits a so-called  \emph{Verdier stratification}, which is perhaps less known
than the Whitney stratification, and is a refinement of the latter
\cite{loi98}.  The key advantage of the Verdier over the Whitney stratification
lies in a Lipschitz-like condition on the (Riemannian) gradients of $f$ on two
adjacent strata, which is established in the paper.
 As the restriction
$f_\cM$ of $f$ to $\cM$ is smooth, the projected iterates, using the Verdier stratification property, are shown to follow a
dynamics which is similar to a (smooth) Stochastic Gradient Descent, up to a
residual term induced by the projection step.  In that sense, the avoidance of
active strict saddles in the non-smooth setting follows from the avoidance of
traps in the smooth setting, as established in~\cite{bra-duf-96}.  We show that
the strict saddle is avoided under the assumption that the (conditional) noise
covariance matrix has a non zero projection on the subspace with negative
curvature associated with $f_\cM$ near the active strict saddle.

Before pursuing, it is important to discuss the matter of the \emph{genericity}
of the assumptions that we just outlined.  First, since our avoidance results
are restricted to the active strict saddles, the question of the presence of
critical points that are neither local minima nor active strict saddles is
immediately raised. Actually, this question was considered in
\cite{drus_iof_lew_generic,dav_dru21}. It is  established there that if $f$ is
definable and weakly convex, then for Lebesgue-almost all vectors $u \in
\bbR^d$, the function $f_u(x) \eqdef f(x)-\scalarp{u}{x}$ admits a finite
number of Clarke critical points, and that each of these points is either an
active strict saddle or a local minimizer.  In that sense, in the class of
definable weakly convex functions, spurious critical points generically
coincide with active strict saddles.  We also need to inspect the generality of the Verdier and the angle conditions. In Theorem~\ref{th:activ_gen_ver} below, we  show
that these assumptions are automatically satisfied when $f$ is weakly convex.
From these considerations, we conclude that generically in the sense of
\cite{drus_iof_lew_generic,dav_dru21}, SGD algorithm~\eqref{eq:sgd-intro}
converges to a local minimum when $f$ is a weakly convex function, assuming
that the noise is omnidirectional enough at the strict saddles.

Let us summarize the contributions of this paper:
\begin{itemize}

\item Firstly, we bring to the fore the fact that definable
functions admit stratifications of the Verdier type. These are more refined
than the Whitney stratifications which were popularized in the optimization
literature by \cite{bolte2007clarke}. While such stratifications are well-known
in the literature on o-minimal structures \cite{loi98}, up to our knowledge,
they have not been used yet in the field of non-smooth optimization. To
illustrate their interest in this field, we study the properties of the Verdier
stratifiable functions as regards their Clarke subdifferentials. Specifically, we refine the so-called projection formula (see \cite[Proposition 4]{bolte2007clarke} and
Lemma~\ref{lm:f_bol_strat} below) to the case of definable, locally Lipschitz continuous functions by establishing a Lipschitz-like condition on the (Riemannian) gradients of two adjacent strata.

\item With the help of the Verdier and the angle conditions, we show that the
SGD avoids the active strict saddles if the noise $\eta_n$ is omnidirectional
enough.

\end{itemize}

Let us mention here a key difference with the first version of the paper submitted to arXiv, where the nonconvergence of SGD toward active strict saddles was stated under a weaker form of angle condition. The latter allowed us to deal with functions that are not (locally) weakly convex. Unfortunately, a serious flaw was pointed out by an anonymous reviewer, who
noted that our proof strongly relied on the use of \cite[Theorem 4.1]{ben-hof-sor-05}, which turns out to be incorrect. We refer to Remark~\ref{rmk:false_first} and Appendix~\ref{app:ben} for a detailed discussion on this issue.

We also point out that, shortly after the first submission of the
present paper, a concurrent and an independent work
\cite{davis_subg_active21} has appeared. In the latter, the question
of the avoidance of active strict saddles by SGD was treated with
very similar techniques. In particular, their proximal aiming
condition coincides with our angle conditions. Other questions, such as
the rate of convergence and the asymptotic normality close to the active
manifolds, were also addressed. We believe that both of our works show
that the Verdier and the angle condition are well-founded and might be
interesting for the optimization community. They, furthermore, open
the way for a more thorough investigation of the avoidance by SGD
of spurious points in a non-smooth context.

The rest of the paper is organized as follows.
Section~\ref{sec-prel} is devoted to the introduction of the mathematical tools
in use in this paper. Most of the results in this section are known, except for
the strengthened projection formula, which is stated in Theorem~\ref{th:f_ver_strat}.
In Section~\ref{sec:act_def_avoid}, we discuss the notion of active strict saddles.
After recalling some results of \cite{dav_dru21}, we introduce the Verdier and angle conditions.
We also discuss the genericity of the these conditions, in the class of weakly convex functions.
In Section~\ref{sec:avoid_traps}, we state the main result of the paper, namely, the avoidance
of active strict saddles. In Section~\ref{sec:fut_work} we discuss possible directions for future work. Sections~\ref{sec:proofs}--\ref{sec:proof_duf} are devoted to the proofs.

\section{Preliminaries}
\label{sec-prel}

\paragraph{Notations.}
Let $d\geq 1$ be an integer.  Given a set $S \subset \bbR^d$,
$\overline{S}$ denotes the closure of $S$, and $\conv(S)$ and
$\overline{\conv}(S)$ respectively denote the convex hull and the
closed convex hull of $S$. The distance to $S$ is denoted as $\dist(x,S) := \inf \{\|y-x\|:y\in S\}$. For a $C^1$ function $g: \bbR^p \rightarrow \bbR^d$ and $x \in \bbR^p$, we denote $J_g(x) \in \bbR^{d \times p}$ the Jacobian of $g$ at $x$.
If $E\subset\bbR^d$ is a vector space, we denote by $P_E$ the $d\times d$ orthogonal
projection matrix onto $E$. We say that a function $f: \bbR^d \rightarrow \bbR$ is weakly convex, if there is $\rho > 0$ such that the function $g(x):= f(x) + \rho \norm{x}^2$ is convex. For two sequences $(a_n), (b_n)$, we write $a_n \gtrsim b_n$ if $ \liminf \frac{a_n}{b_n} >0$. With this notation $a_n \sim b_n$ means $a_n \gtrsim b_n$ and $b_n \gtrsim a_n$.
For $r>0$, $B(0, r)$ denotes the open ball of radius $r$. $(a_n)_{n \geq N}$ denotes a sequence starting from $N\in \bbN$, if there is no possible confusion about the starting index it will also be denoted as $(a_n)$. If $(\mcF_n)$ is a filtration on some probability space and $\eta$ is a random variable, then for $n \in \bbN$, we denote $\bbE[\eta |\mcF_n]$ the conditional expectation of $\eta$ relatively to $\mcF_n$. In the proofs, the latter will usually be denoted as $\bbE_n \eta$.

Throughout the paper, $C$ and $C'$ will refer to positive constants
that can change from line to line and from one statement to another.

\subsection{Functions on Manifolds}\label{sec:subm}

We refer to \cite{Lafontaine_2015, boumal2020intromanifolds} for a detailed introduction on differential geometry.

Given two integers $p\geq 1$ and $k\leq d$,
a $C^p$ map $g \colon U \rightarrow \bbR^{d-k}$ on some open set $U\subset \bbR^d$ is called a $C^p$ submersion
if the rank of $J_g(x)$
is equal to $d-k$ for every $x\in U$.
We say that a set $\cM \subset \bbR^d$ is a $C^p$ submanifold of dimension $k$, if for every $y \in \cM$,
there is a neighborhood $U$ of $y$ and a $C^p$ submersion $g \colon U \rightarrow \bbR^{d-k}$, such that  $U \cap \cM = g^{-1}(\{0\})$.
We represent the tangent space of $\cM$ at $y$ by  $\cT_y\cM:= \ker J_g(y)$ (\emph{n.b.}, the definition is
independent of the choice of $g$).
Equivalently, $\cT_y\cM$ can be represented as the set of vectors $v\in \bbR^d$ such that there exists a differentiable map $c:(-\varepsilon,\varepsilon)\to\bbR^d$ such that
$c((-\varepsilon,\varepsilon))\subset\cM$, $c(0)=y$ and $\dot{c}(0)=v$.

For every $x\in \bbR^d$, we define:
$$
P_M(x) := \arg\min_{y\in M}\|y-x\|\,,
$$
as the (possibly empty) set of points $y^*\in \cM$ such that $\|y^*-x\| = \inf\{\|y-x\|:y\in \cM\}$.
The following lemma can be found in \cite{lew_proj_man08} (see also \cite[Chap. 3, Ex. 24]{Lafontaine_2015}), even though the first part dates back to the 50.
It states that, in the vicinity of any point of $\cM$, $P_M(x)$ is a singleton, henceforth, $P_M$ can be identified to a function on that neighborhood.
Recall that $P_{\cT_y\cM}$ is the orthogonal projection onto $\cT_y\cM$.
\begin{lemma}[Projection onto a manifold]
\label{lm:proj}
Let $\cM$ be a $C^p$ submanifold, with $p\geq 2$. Consider $y\in \cM$. Then, there exists a neighborhood $U$ of $y$,
such that $P_M:U\to\cM$ is a single-valued map. Moreover, $P_M$ is $C^{p-1}$ in that neighborhood, and $J_{P_\cM} = P_{\cT_y\cM}$.
\end{lemma}
We say that a function $f \colon \cM \rightarrow \bbR$ is $C^p$, if $\cM$ is a $C^p$ submanifold, and
if for every $y \in \cM$, there is a neighborhood $U\subset\bbR^d$ of $y$ and a $C^p$ function $F : U \rightarrow \bbR$ that agrees with $f$ on $\cM \cap U$. In this case, $F$ is called a smooth representative of $f$ around $y$ on $M$.
If $f: \cM \rightarrow \bbR$ is $C^1$, we define for every $y\in \cM$,
$$
\nabla_{\cM} f(y) := P_{\cT_y\cM} \nabla F(y)\,,
$$
where $F$ is any smooth representative of $f$ around $y$.
The definition of $  \nabla_\cM f(y) $ does not depend on the choice of $F$
(see e.g. \cite[Section 3.8]{boumal2020intromanifolds}). We refer to $ \nabla_{\cM} f(y)$
as the (Riemannian) gradient of $f$ at $\cM$ (here the Riemannian structure on $\cM$ is implicitely induced from the usual Euclidian scalar product on $\bbR^d$).

If $f:\cM\to\bbR$ is $C^p$, with $p \geq 2$, we define for every $y \in \cM$, the covariant (Riemannian)
Hessian of $f$ at $y$ as the linear operator $\cH_{f,\cM}(y) : \cT_y \cM \rightarrow \cT_y \cM$ given by:
$$
\cH_{f,\cM}(y):\ v\mapsto  P_{\cT_y\cM} J_G(y)v\,,
$$
where $G$ is a $C^{p-1}$ function defined in a neighborhood of $y$ which agrees with $\nabla_\cM f$ on $\cM$,
The definition of $\cH_{f,\cM}(y)$ does not depend on the choice of $G$ (see e.g. \cite[Section 5.5]{boumal2020intromanifolds}).

\subsection{Clarke Subdifferential}

Consider $f: \bbR^d \rightarrow \bbR$ a locally Lipschitz continuous function.
Denote by $\text{Reg}(f)$ the set of points $x$ at which $f$ is differentiable, and by $\nabla f(x)$ the corresponding gradient.,
By Rademacher's theorem, $f$ is differentiable almost everywhere. The \emph{Clarke subdifferential of $f$ at $x$} \cite{cla-led-ste-wol-livre98} is given by:
\begin{equation*}
 \partial f(x) :=  \overline{\conv} \{ v \in \bbR^d: \exists (x_n)\in \text{Reg}(f)^\bbN, (x_n,\nabla f(x_n))\to (x,v)\}\, .
\end{equation*}
That is, $\partial f(x)$ is the closed convex hull of the points of the form $\lim\nabla f(x_n)$ for some sequence $(x_n)$ converging to $x$.
In particular, $\partial f(x)$ simply coincides with $\{ \nabla f(x)\}$ when $f$ is continuously differentiable in a neighborhood of $x$.
We set $\cZ = \{ x\in \bbR^d: 0 \in \partial f(x)\}$. Every point of $\cZ$ is referred to as a Clarke critical point.
In particular, $\cZ$ includes the local minimizers and the local maximizers of $f$.

\begin{definition}[Path-differentiability]\label{def:path}
 A locally  Lipschitz continuous function $f: \bbR^d \rightarrow \bbR$ is said to be \emph{path-differentiable} if
 for every absolutely continuous curve $c:(0,1)\to\bbR^d$, one has for almost every $t\in (0,1)$,
 $$
 (f\circ c)'(t) = \ps{v,\dot{c}(t)},\ \ \forall v\in \partial f(c(t))\,.
 $$
\end{definition}

In non-smooth optimization, the path-differentiability condition is often a crucial hypothesis in order to obtain
relevant results \emph{e.g.}, on the subsequential convergence of iterates \cite{bolte2007clarke,dav-dru-kak-lee-19,bolte2019conservative}.
For instance, a sufficient condition on $f$, which ensure its path-differentiability, is that $f$ is \emph{definable}
w.r.t. an o-minimal structure. We review the concept of o-minimality in the appendix, for the interested reader.
Examples of definable functions include semialgebraic functions, analytic functions on a semialgebraic compact set,
exponential and logarithm (see e.g. \cite{bol_dan_lew09, bier_semi_sub, wil_o_min}).
Moreover, the set of definable functions is closed w.r.t. composition.
In particular, the loss of a neural network is in general a definable function \cite{dav-dru-kak-lee-19}.

 \subsection{Verdier Stratification}

 Let $A$ be a set in $\bbR^d$, a $C^p$ stratification of $A$ is a finite partition of $A$ into a family of \emph{strata} $(S_i)$ such that each of the $S_i$ is a $C^p$ submanifold, and such that:
 \begin{equation*}
   S_i \cap \overline{S}_j \neq \emptyset \implies S_i \subset \overline{S}_j \backslash S_j \, .
 \end{equation*}
 Given a family $\{ A_1, \dots, A_k\}$ of subsets of $A$, we say that a stratification $(S_i)$ is \emph{compatible with} $\{ A_1, \dots, A_k\}$, if each of the $A_i$ is a finite union of strata.
 We say that a stratification $(S_i)$ is definable, if every stratum $S_i$ is definable w.r.t. some o-minimal structure (see Appendix~\ref{sec:o-min}).
 If  $E_1, E_2$ are two vector spaces such that  $E_1 \neq \{ 0\}$, we define:
  \begin{equation}\label{eq:angle_vect}
   \bsd_a(E_1, E_2) = \sup_{u \in E_1, \norm{u} = 1} \dist(u, E_2)\, ,
 \end{equation}
 and we set $\bsd_a(\{ 0\}, E_2) = 0$.
 Note that $\bsd_a(E_1, E_2) = 0$ implies $E_1 \subset E_2$.
 \begin{definition}
   Let $(S_i)$ be a $C^p$ stratification of some set $A \subset \bbR^d$.
 We say that $(S_i)$ satisfies the \emph{Verdier property-(v)},
 if for every couple of distinct strata $S_i, S_j$ such that $S_i \cap \overline{S_j} \neq \emptyset$ and for each $y \in S_i$, there are two positive constants $\delta, C$ such that:
 \begin{equation}\label{eq:strat_verd}
   \begin{array}{ll}
     y' &\in B(y, \delta) \cap S_i \\
     x &\in B(y, \delta) \cap S_j
   \end{array} \implies \bsd_a(\cT_{y'} S_i, \cT_{x} S_j) \leq C \norm{y' - x} \, .
 \end{equation}
 In this case, we refer to $(S_i)$ as a Verdier $C^p$ stratification of $A$.
\end{definition}
A Verdier stratification is a special case of a Whitney stratification (we refer the reader to Appendix~\ref{sec:whitney}
for a review on Whitney stratifications).
Whereas the Whitney stratification can now be considered as well known in optimization community, the Verdier stratification is comparatively less popular.
In the framework of nonsmooth optimization, one of the main interests of the Verdier stratification over the Whitney stratification,
is given by Theorem~\ref{th:f_ver_strat} below,
which is one of the contributions of this paper.
Theorem~\ref{th:f_ver_strat} can be seen as a strengthening of the so-called ``projection formula'' for Whitney stratifiable functions,
which we recall in Lemma~\ref{lm:f_bol_strat} of Appendix~\ref{sec:whitney}, for the sake of completeness.

Before stating this result, we make two important remarks.
First, any locally Lipschitz continuous function $f: \bbR^d \rightarrow \bbR$ whose graph
admits a Verdier stratification, is path-differentiable in the sense of Definition~\ref{def:path}.
This is a consequence of Lemma~\ref{lem:whitney-implest-pathdiff} in Appendix~\ref{sec:o-min}.
Second, if $f$ is definable w.r.t. an o-minimal structure (see Appendix~\ref{sec:o-min}), then, for every $p\geq 1$, its graph
admits a Verdier $C^p$ stratification. This is a consequence of the following fundamental result.

\begin{proposition}[{\cite[Theorem 1.3]{loi98}}]\label{prop:Verdier-existe}
 Let $\{A_1, \dots, A_k\}$ be a family of definable sets of $\bbR^d$. For any $p \geq 1$, there is a Verdier $C^p$ stratification of $\bbR^d$ compatible with $\{A_1, \dots, A_k \}$.
\end{proposition}
Finally, we state the main contribution of this section, which can be interpreted as a strengthened version of
the projection formula (Lemma~\ref{lm:f_bol_strat}).

\begin{theorem}[Strengthened projection formula]\label{th:f_ver_strat}
     Let $f: \bbR^d \rightarrow \bbR $ be a definable, locally Lipschitz continuous function. Let $p$ be a positive integer. There is $(X_i)$, a definable Verdier $C^p$ stratification of $\bbR^d$, such that $f$ is $C^p$ on every stratum and for every couple of distinct strata $X_i, X_j$ such that $X_i \cap \overline{X_j} \neq \emptyset$ and for every $y \in X_i$, there is $C,\delta >0$, such that for any two points $y' \in  B(y,\delta) \cap X_i$, $x \in B(y, \delta) \cap X_j$,
     \begin{equation}\label{eq:ver_grad}
       \norm{P_{\cT_{y'}X_i}(\nabla_{X_j} f(x)) - \nabla_{X_i} f(y')} \leq C \norm{x- y'} \, ,
     \end{equation}
     and, moreover, for any $x \in B(y, \delta) \backslash X_i $ and any $v \in \partial f(x)$,
     \begin{equation}\label{eq:verd_subgrad}
       \norm{P_{\cT_{y'}X_i}(v) - \nabla_{X_i} f(y')} \leq C \norm{x - y'} \, .
     \end{equation}
\end{theorem}
\begin{proof}
  In this proof $C' >0$ will denote some constant that can change from line to line.
Consider $(S_i)$ and $(X_i)$ as in Lemma~\ref{lm:f_bol_strat}. We claim that for any index $j$ and $x \in X_j$, we have $\cT_{x, f(x)}S_j = \{ (h, \scalarp{\nabla_{X_i}f(x)}{h}) : h \in \cT_x X_j\}$. Indeed, consider $(h_x, h_f) \in \cT_{x, f(x)}S_j$ and a $C^p$ curve $c \colon (-\varepsilon, \varepsilon) \to \bbR^d$
s.t. $c((-\varepsilon, \varepsilon)) \subset S_j$, $c(0) = (x,f(x))$ and $\dot{c}(0) = (h_x, h_f)$. Consider a $C^p$ function $F$ that agrees with $f$ on $X_j$, then $(c_x(t), c_f(t)) = (c_x(t), F(c_x(t)))$ and we have $\dot{c}_x(0) = h_x$ and $\dot{c}_f(0) = \scalarp{\nabla F(x)}{h_x} = \scalarp{\nabla_{X_j}f(x)}{h_x}$.

Consider $(S_i')$ a Verdier stratification of $\graph (f)$ compatible with $(S_i)$. Then the projection of $S'_i$ onto its first $d$ coordinates, that we denote $X'_i$, is still a submanifold s.t. $f$ is $C^p$ on $X'_i$.
Consider $(y, f(y)) \in S'_i$, $S'_j$ a neighboring stratum and $C, \delta$ as in Equation~\eqref{eq:strat_verd}. Denote by $L$ the Lipschitz constant of $f$ on $B(y, \delta)$ and $\delta' = \frac{\delta}{L+1}$. Then, for every $x \in B(y, \delta')$, we have:
\begin{equation*}
\norm{(y, f(y)) - (x, f(x))} \leq (1+L)\norm{y - x} \leq \delta \, ,
\end{equation*}
that is to say $(x, f(x)) \in B((y, f(y)), \delta)$.

Consider $y' \in X'_i \cap B(y, \delta')$, $x \in X'_j \cap B(y, \delta')$ and $h_{y'} \in \cT_{y'}X'_i$ with $\norm{h_{y'}} = 1$. We have that  $(h_{y'}, \scalarp{\nabla_{X'_i} f(y')}{h_{y'}}) \in \cT_{(y', f(y'))}   S'_i$ and by the Verdier's condition there is $h_x \in \cT_{x}X'_j$ s.t.
\begin{equation*}
\norm{\frac{1}{c_h}\left(h_{y'}, \scalarp{\nabla_{X'_i} f(y')}{h_{y'}}\right) - (h_x , \scalarp{\nabla_{X'_j} f(x)}{h_x})} \leq C(L+1) \norm{x - y'} \, ,
\end{equation*}
where $c_h = \norm{(h_{y'}, \scalarp{\nabla_{X'_i} f(y')}{h_{y'}})} \leq C'$. Therefore,
\begin{equation*}
\norm{h_{y'}- c_h h_x} \leq C' \norm{x- y'} \quad \textrm{ and } \quad \norm{ c_h\scalarp{\nabla_{X'_j} f(x)}{h_x} - \scalarp{\nabla_{X'_i} f(y')}{h_{y'}}} \leq C' \norm{x - y'} \, .
\end{equation*}
Thus, it holds that:
\begin{equation}\label{eq:pf_ver_gra}
  \begin{split}
    \norm{\scalarp{\nabla_{X'_j} f(x)}{h_{y'} - c_h h_x} } + \norm{ c_h\scalarp{\nabla_{X'_j} f(x)}{h_x} - \scalarp{\nabla_{X'_i} f(y')}{h_{y'}}}
    \leq C' \norm{x - y' } \, .
  \end{split}
\end{equation}
 Noticing that by the projection formula, for all $v \in \partial f(x)$, it holds that $\scalarp{v}{ h_x} = \scalarp{\nabla_{X'_j} f(x)}{h_x}$, we also obtain for such a $v$,
 \begin{equation}\label{eq:pf_ver_sub}
   \begin{split}
     \norm{\scalarp{v}{h_{y'} - c_h h_x} } + \norm{ c_h\scalarp{v}{h_x} - \scalarp{\nabla_{X'_i} f(y')}{h_{y'}}}
     \leq C' \norm{x - y' } \, .
   \end{split}
 \end{equation}
Thus, using Equation~\eqref{eq:pf_ver_gra} and applying a triangle inequality, we obtain:
\begin{equation}\label{eq:pf_ver_gra_end}
  \norm{\scalarp{\nabla_{X'_j} f(x) - \nabla_{X'_i} f(y')}{ h_{y'}}} \leq C' \norm{x - y'} \, .
\end{equation}
And, similarly, using Equation~\eqref{eq:pf_ver_sub},
\begin{equation}\label{eq:pf_ver_sub_end}
\forall v \in \partial f(x)\, , \quad  \norm{\scalarp{v - \nabla_{X'_i} f(y')}{ h_{y'}}} \leq C' \norm{x - y'} \, .
\end{equation}

The proof is now completed by noticing that $h_{y'} \in \cT_{y'}X'_i$ was an arbitrary vector of unitary norm and that one can choose $C', \delta$ such that Equations~\eqref{eq:pf_ver_gra_end} and \eqref{eq:pf_ver_sub_end} hold uniformly for all $x \in B(y, \delta) \cap X_j'$, where $X_j'$ is any stratum neighboring $X_i'$.

\end{proof}

\begin{remark}\label{rmk:cons_field_ver}
  Equation~(\ref{eq:verd_subgrad}) in Theorem~\ref{th:f_ver_strat} remains true if the Clarke subdifferential $\partial f$
  is replaced by any, definable, set-valued map $D: \bbR^d \rightrightarrows \bbR^d$ which is a so-called
  \emph{conservative field for the potential $f$} \emph{i.e.}, the condition $v\in \partial f(x)$ can be replaced
  by $v\in D(x)$.
  The concept of conservative fields was introduced in \cite{bolte2019conservative},
  in order to circumvent the fact that automatic differentiation
  procedures such as those used in Tensorflow,
  do not necessarily produce Clarke subgradients. The Clarke subdifferential $\partial f$ is, among others,
  one example of a conservative field.
  It was shown in \cite[Theorem 4]{bolte2019conservative} that the projection formula (Lemma~\ref{lm:f_bol_strat}
  in this paper) still holds, if $\partial f$ is replaced by a definable conservative field $D$.
  Using this generalization in the proof of Theorem~\ref{th:f_ver_strat} instead of Lemma~\ref{lm:f_bol_strat},
  one conclude that the same generalization holds for our strengthened projection formula.
\end{remark}

\section{Active Strict Saddles}\label{sec:act_def_avoid}

In this section, $f : \bbR^d \rightarrow \bbR$ is supposed to be a
locally Lipschitz continuous function.  We recall the definition
$\cZ:=\{x\in \bbR^d:0\in \partial f(x)\}$.

\subsection{Definition and Existing Results}

Let $p\geq 2$ be an integer.
\begin{definition}[Active manifold\footnote{We must notice here that the notion of an active (or, as it sometimes referred to, identifiable) manifold is closely related to the notion of partial smoothness introduced in \cite{Lewis2002ActiveSN}. Indeed, as it was shown in \cite[Proposition 8.4]{drus_lew_optim_sens} both are equivalent under a non-degeneracy condition: $0$ is in the relative interior of the proximal subdifferential of $f$ at $x^*$.},
 \cite{drus_lew_optim_sens}]\label{def:act_man}
   Consider $x^*\in \cZ$.
   A set $\cM\subset \bbR^d$ is called a \emph{$C^p$ active manifold  around $x^*$}, if there is a neighborhood $U$ of $x^*$ such that the following holds.
   \begin{enumerate}[i)]
     \item \textbf{Smoothness condition:} $\cM \cap U$ is a $C^p$ submanifold and $f$ is $C^p$ on $\cM \cap U$.
     \item \textbf{Sharpness condition:}
     \begin{equation*}
       \inf\{ \norm{v}: v \in \partial f(x), x \in U \backslash \cM \} > 0\, .
     \end{equation*}
   \end{enumerate}
 \end{definition}

 \begin{definition}[Active strict saddle]\label{def:act_str}
   We say\footnote{The definition of active strict saddles provided
     in \cite{dav_dru21} involves the notion of parabolic
     subderivatives.  In this paper, we found convenient to use the
     equivalent Definition \ref{def:act_str}, which is closer in
     spirit to notions of differential geometry.}
   that a point $x^* \in \cZ$ is an \emph{active strict saddle (of order $p$)} if there exists a $C^p$ active manifold $\cM$ around $x^*$,
   and a vector $w \in \cT_{x^*} \cM$, such that $\nabla_{\cM}f(x^*) = 0$ and $ \scalarp{w}{\cH_{f,\cM}(x^*)(w)}< 0$.\\
   We say that $f$ satisfies the \emph{active strict saddle
     property (of order $p$)}, if it has a finite number of Clarke critical points, and each of these points is either an active
   strict saddle of order $p$ or a local minimizer.
 \end{definition}
In the special case of a \textbf{smooth} function $f$, the space $M=\bbR^d$ is trivially an active manifold around any critical point $x^*$ of $f$.
If $x^*$ is moreover a \emph{trap} in the sense provided in the introduction (\emph{i.e.}, the Hessian matrix of $f$ at $x^*$ admits a negative eigenvalue),
then $x^*$ is trivially an active strict saddle. Hence, the smooth setting can be handled as a special case.

 The archetype of an active strict saddle is given by the following example.

\begin{example}\label{ex:act_strict}
 The point $(0,0)$ is an active strict saddle of the function $f : \bbR^2 \rightarrow \bbR$ given by $f(y,z) = - y^2 + |z|$. Indeed,
 \begin{equation*}
   \partial f((y,z)) =  \begin{cases} \{(-2y, 1)\} \textrm{ if } z >0 \, ,\\
   \{(-2y, -1)\} \textrm{ if } z < 0 \, ,\\
  \{ -2y\} \times [-1, 1] \textrm{ otherwise } \, ,
 \end{cases}
\end{equation*}
and the set $\cM = \bbR \times \{ 0\}$ is a $C^{2}$ active manifold.
Moreover, $\nabla_{\cM}f((y, 0)) = (-2y, 0)$ and the scalar product between $(1,0)$ and $\cH_{f,\cM}(0)((1, 0))$ is equal to $-2$.
\end{example}
While the definition of an active strict saddle might seem peculiar at first glance, the
following proposition of Davis and Drusvyatskiy shows that a generic
definable and weakly convex function satisfies a strict saddle
property. The proof is grounded in the work of \cite{drus_iof_lew_generic}.

 \begin{proposition}[{\cite[Theorem 2.9]{dav_dru21}}]\label{prop:activ_gen}
   Assume that $f$ is definable and weakly convex.
   Define $f_u(x) := f(x) - \scalarp{u}{x}$, for every $u\in \bbR^d$.
   Then, for every $p\geq 2$ and for Lebesgue-almost every $u\in \bbR^d$, $f_u$ has the active strict saddle property of order $p$.
 \end{proposition}
 It is worth noting that the result of \cite[Theorem 2.9]{dav_dru21} is in fact a bit stronger than Proposition~\ref{prop:activ_gen},
 because it states moreover that for almost all $u$, the cardinality of the set of Clarke critical points of $f_u$ is upper bounded
 by a finite constant which depends only on $f$.

 One can wonder if Proposition~\ref{prop:activ_gen} may still hold if $f$ is definable and locally Lipschitz, but not weakly convex.
The answer is negative, as shown by the following example.
 \begin{example}\label{ex:non-wcvx}
   Let $f: \bbR^2 \rightarrow \bbR$ be defined as $f(y,z) = -|y| + |z|$. Then for any $u \in B(0, 1)$, $(0,0)$ is a critical point for $f_u$, but is neither a local minimum nor an active strict saddle.
 \end{example}

 \subsection{Verdier and Angle Conditions}\label{sec:ver_ang}

 On the top of the items {\it i-ii)} of Definition~\ref{def:act_man},
 we introduce the following useful conditions.

 \begin{definition}\label{def:ver_ang}
   Let $\cM$ be a $C^2$ active manifold around some
   $x^*\in\cZ$. We say that $\cM$ satisfies the Verdier condition
   and the angle condition, if there is $U$ a neighborhood of $x^*$ such that the following conditions hold respectively.
\begin{enumerate}[i)]\setcounter{enumi}{2}
     \item\label{verd}\textbf{Verdier condition.} There is $C \geq 0$, such that for  every $y \in \cM \cap U$ and every $x \in U$,
     \begin{equation*}
        \quad \norm{P_{\cT_{y}\cM} (v) - \nabla_{\cM}f(y) } \leq C \norm{x - y}, \quad{}  \forall v \in \partial f(x)\, .
     \end{equation*}\\
   \item\label{angle} \textbf{Angle condition.}
     There is $\beta >0$, such that for every $x \in U$ and for every $v \in \partial f(x)$,
     \begin{equation*}
       \scalarp{v}{x -P_{\cM}(x)} \geq \beta \norm{x -P_{\cM}(x)}
  \, .
     \end{equation*}
   \end{enumerate}
 \end{definition}
 \begin{definition}
   An active strict saddle $x^*$ is said to satisfy the \emph{Verdier and angle conditions}, if the
   active manifold $M$ in Definition~\ref{def:act_str} satisfies the Verdier and angle conditions.
   The function $f$ is said to satisfy the \emph{active strict saddle property of order $p$ with the Verdier and angle conditions}, if
   it satisfies the active strict saddle property of order $p$ and if every active strict saddle satisfies the Verdier and angle conditions.
 \end{definition}
 The Verdier condition merely states that $\cM$ is one of the stratum
 of the Verdier stratification of Theorem~\ref{th:f_ver_strat}.
The purpose of the angle condition is to ensure that, close to $\cM$, the subgradients of $f$ at $x$ are always directed outwards of $\cM$.
The latter will allow us to
prove that the iterates of SGD  converge to
$\cM$ fast enough. In the concurrent work of \cite{davis_subg_active21} these conditions were named as \textit{strong (a)} and \textit{proximal aiming} conditions.

The following theorem strengthens the genericity result of Proposition~\ref{prop:activ_gen} by establishing that the active strict saddle property with the Verdier and angle conditions is satisfied by a generic definable and weakly convex function. We recall the notation $f_u(x) = f(x)-\ps{u,x}$.
\begin{theorem}\label{th:activ_gen_ver}
 Assume that $f : \bbR^d \rightarrow \bbR$ is a definable, weakly
 convex function. For every $p \geq 2$, and for Lebesgue-almost every $u\in \bbR^d$,
 $f_u$ satisfies the active strict saddle property of order $p$ with the
 Verdier and angle conditions.
\end{theorem}
\begin{proof}
   Let $\{X_1, \dots, X_k \}$ be the $C^p$ Verdier stratification from Theorem~\ref{th:f_ver_strat}. Upon noticing that in the proof of \cite[Corollary 4.8 and Theorem 4.16]{drus_iof_lew_generic} the active manifold \footnote{The name \emph{active manifold} follows from the work of \cite{dav_dru21}, while in \cite{drus_iof_lew_generic} they are called identifiable manifolds. Both terminologies are usual and go back at least to \cite{burke_more_88, wright_93}.} can be chosen adapted to $\{X_1, \dots, X_k \}$, the existence of an active manifold with a Verdier condition follows from \cite[Theorem 2.9, Appendix A]{dav_dru21}. To prove the angle condition note that by weak convexity of $f$ there is $\rho \geq 0$ such that:
 \begin{equation*}
   f(P_{\cM}(x)) - f(x) \geq \scalarp{v}{P_{\cM}(x) - x} - \rho \norm{x - P_{\cM}(x)}^2 \quad \forall v \in \partial f(x) \, .
 \end{equation*}
 Furthermore, it was noticed in \cite[Theorem D.2]{dav_dru_cha19} that weak convexity of $f$ implies the existence of $\alpha >0$ such that for $x$ close enough to $x^*$, it holds:
 \begin{equation*}
   f(x) \geq f(P_{\cM}(x)) + \alpha \norm{P_{\cM}(x) -x} \, .
 \end{equation*}
 Combining both inequalities, we obtain:
 \begin{equation*}
   \forall v \in \partial f(x), \quad \scalarp{v}{x -P_{\cM}(x)} \geq \alpha \norm{x- P_{\cM}(x)} - \rho \norm{x - P_{\cM}(x)}^2 \, .
 \end{equation*}
 Taking $U = B(x^*, r)$, with $r$ small enough, we see that the angle condition is satisfied.
\end{proof}

\begin{remark}
 Let $M$ be an active manifold around $x^*$. It is clear from the proof of Theorem~\ref{th:activ_gen_ver}, that when $f$ is weakly convex, $M$ always satisfies the angle condition.
 Otherwise stated, the angle condition is simply \emph{true} in case of weakly convex functions.
 One may wonder if there are examples of (non-weakly convex) functions that have active strict saddles \emph{without} the angle condition. The following generic example exhibits one of those.
\end{remark}
\begin{example}\label{ex:non-angle}
 The function $f: \bbR^2 \rightarrow \bbR$ given by $f(y,z) = -y^2 - |z|$ is not weakly convex.
 Its unique Clarke critical point $(0, 0)$ is an active strict saddle, satisfying the Verdier condition but \emph{not} satisfying the angle condition. Notice, furthermore, that this example is generic in the following sense: if $\norm{u} < 1$, then $(0,0)$ is also an active strict saddle of $f_u$, satisfying the same properties.
\end{example}

\section{Avoidance of Active Strict Saddles}\label{sec:avoid_traps}

Let $f : \bbR^d \rightarrow \bbR$ be a locally Lipschitz continuous function.
On a probability space $(\Omega, \mcF, \bbP)$,
consider a random variable $x_0$ and random sequences $(v_n)$, $(\eta_{n})$ on $\bbR^d$.
Define the iterates:
\begin{equation}\label{eq:sgd}
 x_{n+1} = x_n - \gamma_n v_n + \gamma_n \eta_{n+1} \, ,
\end{equation}
where $(\gamma_n)$ is a deterministic sequence of positive numbers.
Let $(\mcF_n)$ be a filtration on $(\Omega, \mcF, \bbP)$.
\begin{assumption}\label{hyp:model}\-
 \begin{enumerate}[i)]
   \item The function $f$ is path differentiable.
   \item For every $n \in \bbN$, $v_n\in \partial f(x_n)$.
   \item The sequences $(v_n)$, $(\eta_n)$ are adapted to $(\mcF_n)$, and $x_0$ is $\mcF_0$-measurable.
   \item\label{hyp:steps} There are constants $c_1, c_2  >0$ and $\alpha \in (1/2, 1]$ s.t. for all $n \geq 1$:
   \begin{equation*}
     \frac{c_1}{n^\alpha}\leq \gamma_n \leq \frac{c_2}{n^\alpha} \, .
   \end{equation*}
 \end{enumerate}
\end{assumption}
Consider a point $x^*\in \cZ$.
\begin{assumption}\label{hyp:activ_strict}
 The point $x^*$ is an active strict saddle of order 5 satisfying the Verdier and angle conditions.
\end{assumption}

  The reason for which we posit a smoothness condition of order 5 (while the definition of an active strict saddle
  only requires the order 2) will be clear from Section~\ref{sec:proof_main}. It is related to the fact that, in our proof,
  $\nabla (F \circ P_{\cM})$ need to be $C^3$, where $F$ is any smooth representative of $f$ on $\cM$.
  By Lemma~\ref{lm:proj} this will be obtained as soon as $\cM$ is a $C^5$ manifold.

  We state our last assumption.
Interpreting the map $\cH_{f, \cM}(x^*)\circ P_{\cT_{x^*} \cM}$ as a quadratic form on $\bbR^d$,
we write $\bbR^d = E^{-} \oplus E^{+}$, where $E^{-}$
(respectively $E^{+}$) is the vector space spanned by the eigenvectors
that have negative (respectively nonnegative) eigenvalues. By results of
Section~\ref{sec:subm},  $E^{-} \subset \cT_{x^*} \cM$. Moreover, by
Assumption~\ref{hyp:activ_strict}, we note that $\dim E^{-} \geq 1$.

\begin{assumption}\label{hyp:noise_exit}
 The following holds almost surely on the event $[x_n \rightarrow x^*]$.
 \begin{enumerate}[i)]
   \item $\bbE[\eta_{n+1} | \mcF_n] = 0$, for all $n$.
   \item $\limsup \bbE[ \norm{\eta_{n+1}}^4 | \mcF_n]< + \infty$.
   \item
     Denote $\eta_{n+1}^{-}$ the projection of $\eta_{n+1}$ onto $E^-$. We have:
   \begin{equation*}
     \liminf\bbE[\norm{\eta_{n+1}^{-}} | \mcF_n] > 0
 \end{equation*}
 \end{enumerate}
\end{assumption}

\begin{remark}
We discuss Assumption~\ref{hyp:noise_exit}. The first point is standard. The third point ensures that $\eta_{n}$ explores the negative curvature space $E^-$, so that the iterates must eventually escape $x^*$,
as will be explained in Section~\ref{sec:proof_main}.
The second point deserves more comments.  In order to establish the convergence of SGD toward the set $\cZ$,
the assumption $\limsup \bbE[ \norm{\eta_{n+1}}^2 | \mcF_n]< + \infty$ is standard.
However, in order to establish the avoidance of spurious critical point, this assumption should be strengthened by requiring the boundedness of the (conditional) fourth order moments.
Comparatively to this paper, \cite{pem-90, ben-(cours)99} make the stronger assumption that the sequence $(\eta_{n})$ is by a deterministic constant,
whereas the  concurrent work of \cite{davis_subg_active21}, which has been submitted shortly after the present paper, assumes
a boundedness condition on eight-order moments.
Finally, although  \cite{bra-duf-96} claims that only second order moments need to be bounded, a careful examination of their proof
reveals that fourth order moments are actually needed.
\end{remark}

We are ready to state the following theorem, which is the main result of this paper. Its proof is devoted to Section~\ref{sec:proof_main}.
\begin{theorem}\label{th:avoid_trap}
 Let Assumptions~\ref{hyp:model}--\ref{hyp:noise_exit} hold. Then $\bbP(x_n \rightarrow x^*) = 0$.
\end{theorem}

Combining Theorem~\ref{th:avoid_trap} with the results of Section~\ref{sec:ver_ang} we obtain that, under appropriate assumptions, SGD on a generic definable, weakly convex function converges to a local minimizer. We state this result in the following corollary.
\begin{corollary}\label{cor:conv_loc_min}
 Let Assumptions~\ref{hyp:model} and \ref{hyp:activ_strict}
 hold. Assume that $f$ has the active strict saddle property of order
 5 with the Verdier and angle conditions. Moreover, assume that, almost surely, the following holds.
 \begin{enumerate}[i)]
   \item $\bbE[\eta_{n+1} | \mcF_n] = 0$, for all $n \in \bbN$.
   \item For every $C >0$, \begin{equation*}
     \limsup \bbE[ \norm{\eta_{n+1}}^4 | \mcF_n] \1_{\norm{x_n}\leq C} < + \infty \, .
 \end{equation*}
 \item For all $w\in \bbR^d\backslash \{0\}$,
 \begin{equation*}
   \liminf\bbE[|\ps{w,\eta_{n+1}}
  |\, | \mcF_n] > 0\,.
 \end{equation*}
 \end{enumerate}
 Then, almost surely, the sequence $(x_n)$ is either unbounded, or
 converges to a local minimizer of $f$.
\end{corollary}

\begin{proof}
  We will first show that, almost surely, if $(x_n)$ is bounded, then it converges to a unique point $z \in \cZ$. Indeed, denote $A \in \mcF$ the event on which $(x_n)$ is bounded and notice that $A = \bigcup_{C \in \bbN}A_C := \bigcup_{C \in \bbN}[\sup_{n \in \bbN} \norm{x_n} \leq C]$.
   For $n \in \bbN$, define $\tilde{\eta}_{n+1} = \eta_{n+1} \1_{\norm{x_n} \leq C}$ and notice that on $A_C$, $\eta_{n+1} = \tilde{\eta}_{n+1}$. By a standard Martingale argument $\sum_{i=0}^{+\infty} \gamma_i \tilde{\eta}_{i+1}$ converges, and $\sup_{n \in \bbN} \norm{v_n} < + \infty$. Therefore, on $A_C$, we can apply \cite[Theorem 3.2]{dav-dru-kak-lee-19} and obtain that the limit points of $(x_n)$ are all lying in $\cZ$.
  Furthermore, on $A_C$, it holds that:
  \begin{equation*}
    \norm{x_{n+1} - x_n} \leq \gamma_n (L + \norm{\tilde{\eta}_{n+1}}) )\, ,
  \end{equation*}
  where $L$ is the Lipschitz constant of $f$ on $\overline{B(0, C)}$. Since the right-hand side of this inequality goes to zero, this implies that $\norm{x_{n+1} - x_n} \rightarrow 0$. The latter, along with the boundedness of $(x_n)$, implies that the set of limit points of $(x_n)$ is a connected set. From the active strict saddle property of $f$ we know that $\cZ$ is finite and the only connected set of a finite set is a unique point. This implies that, on $A_C$, $(x_n)$ converges to a unique point. Since $C \in \bbN$ was arbitrary, this remains true on $A$.

  Finally, let $x^* \in \cZ$ be an active strict saddle of $f$. Assumptions of the corollary immediately implies Assumption~\ref{hyp:noise_exit}. Hence, $\bbP(x_n \rightarrow x^*) = 0$. Thus, as soon as $(x_n)$ is bounded, it converges to a unique point, which is not an active strict saddle. By the active strict saddle property of $f$, this implies the convergence of $(x_n)$ to a local minimum of $f$.

\end{proof}

\section{Further Topics}\label{sec:fut_work}

\subsection{Conservative Fields and Empirical Risk Minimization}

  In several situations such as machine learning, one is interested in the scenario where $f$ is defined as a finite sum of the form:
  $$
  f(x) = \frac 1N \sum_{i=1}^N f_i(x)
  $$
  for some (nondifferentiable) locally Lipschitz functions $f_1,\dots,f_N$. In the machine learning community, this problem
  is referred to as \emph{empirical risk minimization}, and the algorithm referred to as SGD,
  consists in drawing an independent sequence $(I_n)$, uniformly chosen in $\{1,\dots,N\}$, and to define the iterates:
  \begin{equation}
    \tilde x_{n+1} = \tilde x_n - \gamma_n v_{n+1}\,,
\label{eq:dzaelikj}
  \end{equation}
  where $v_{n+1}$ is one element of the Clarke subdifferential of $f_{I_{n+1}}$ at point $x_n$, or, even more generally,
  $$
  v_{n+1} \in D_{I_{n+1}}(\tilde x_n)\,,
  $$
  where $D_1,\dots, D_N:\bbR^d\rightrightarrows \bbR^d$ are \emph{conservative fields} for the potentials $f_1,\dots, f_N$ respectively
  (we refer to \cite{bolte2019conservative} for a complete review of conservative fields).
  As a matter of fact,  the iterates~(\ref{eq:dzaelikj}) do not in general satisfy the inclusion~(\ref{eq:sgd-intro}).
  Hence, our paper does not, strictly speaking, encompass the algorithm~(\ref{eq:dzaelikj}).
  As far as the generalization of our results is concerned, two approaches can be used.

  A first approach follows the idea of \cite{bia-hac-sch-SVA, bolt_pauw_nips20, bolt_tam_pauw22_sampl}.
  It is based on the fact that the Clarke subdifferential (or more generally, any conservative field) coincides almost everywhere
  with the gradient (see \cite{bolte2019conservative}).
  That is, for every $i=1,\dots,N$ and for Lebesgue-almost every $x$, $D_i(x)=\{\nabla f_i(x)\}$.
  Under the assumption that the initialization $\tilde x_0$ is randomly chosen according to some probability density function on $\bbR^d$,
  and provided that every step size $\gamma_n$ lies outside a certain Lebesgue-negligible set, every random variable $\tilde x_n$ can be shown to
  admit a density w.r.t. the Lebesgue measure. This implies that $D_i(\tilde x_n)=\{\nabla f_i(\tilde x_n)\}$ almost surely, for every $i$.
  Hence $v_{n+1} = \nabla f_{I_{n+1}}(\tilde x_n)$.
  As a consequence, the iterates $(\tilde x_n)$ satisfy w.p.1:
  $$
  \tilde x_{n+1} = \tilde x_n  - \gamma_n \nabla f(\tilde x_n) + \gamma_n \eta_{n+1}\,,
  $$
  where  $\eta_{n+1} := \nabla f_{I_{n+1}}(\tilde x_n) - \nabla f(\tilde x_n)$ is a martingale increment sequence. Otherwise stated, under these assumptions,
  $\tilde x_n$ satisfy Equation~(\ref{eq:sgd-intro}) w.p.1., for every $n$. Our conclusions can therefore be extended to the algorithm~(\ref{eq:dzaelikj}),
  provided that the initialization is random, and under mild assumptions on the step sizes.
  We note that a similar conclusion can be also obtained if, instead of randomizing the initial point,
  one artificially adds a small random perturbation (e.g. Gaussian) to the update Equation~(\ref{eq:dzaelikj}), in order to ensure that every $\tilde x_n$ admits a density.

  A second approach consists in working in the framework of conservative set-valued fields, without additional assumptions such as the above randomized initialization.
  As a matter of fact, the iterates~(\ref{eq:dzaelikj}) can be written under the form
  $\tilde x_{n+1} = \tilde x_n  - \gamma_n D(\tilde x_n) + \gamma_n \eta_{n+1}$, where $D=N^{-1}\sum_{i=1}^N D_i$ is, again, a conservative field, and $(\eta_{n})$
  is a martingale increment sequence.
  Previous works such as \cite{bolte2019conservative} consider this kind of dynamics, and establish the convergence
  to the set of zeroes of $D$.
  Furthermore, a strengthened projection formula such as the one of Theorem 1, also holds when substituting $D$ with $\partial f$.
  This implies that, generically, every zero of $D$ for definable $f$  belongs to an active manifold satisfying a Verdier condition.
  Unfortunately, the main issue is that the set of zeroes of $D$ can be substantially larger than the set of Clarke critical points.
  In particular, even for weakly convex functions $f$, it cannot be established that, generically, the spurious zeroes of $D$
  (\emph{i.e.}, the ones which are not local minimizers of $f$) are active strict saddles. For instance,
  the following example is generic: $f: \bbR \rightarrow \bbR$, with $f(x) =x $, $D(0) = [0,1]$ and $D(x) =\{ 1\}$ if $x \neq 0$.
  A consequence is that Corollary~\ref{cor:conv_loc_min} cannot be immediately generalized to conservative fields.
  A characterization of the spurious zeroes of $D$, and the proof of the fact that SGD avoids such points, constitute interesting problems for future researches.

\subsection{Beyond Weak Convexity}

Active strict saddles are generic in case of
(definable) weakly convex functions. As a consequence,
our proof of nonconvergence to active saddles (Theorem~\ref{th:avoid_trap})
allows to conclude that, in case of generic weakly convex functions, SGD converges to a local minimizer (Corollary~\ref{cor:conv_loc_min}).
A natural question is to know whether the same conclusion holds
in the absence of weak convexity.

To that end, it is mandatory to characterize, aside from active strict
saddles, the generic critical points of definable but not necessarily
weakly convex functions. Such a characterization was
provided in the Ph.D. manuscript of the third author \cite{schechtman_phd}.
It is established that, generically, any Clarke critical point of a definable function
is located on an active manifold $\cM$ (satisfying a Verdier condition),
and is either a local minimum, an active strict saddle or
what the author of \cite{schechtman_phd} called a \emph{sharply repulsive critical point},
namely, a point $x^*$ which is a local minimizer of $f_\cM$, and such that
the subgradients of $f$ in the vicinity of $x^*$ are pointing toward $\cM$
(see Example~\ref{ex:non-wcvx} for an illustration of such a point).

An important but challenging problem is to establish, in the absence of weak convexity, that both active strict saddles and sharply repulsive
points are avoided.
\begin{remark}\label{rmk:false_first}
  In case of generic definable but non-weakly convex functions,
  the nonconvergence of SGD toward active strict saddles was incorrectly stated in the first arXiv version of this paper. A serious flaw in the proof was pointed out by one the anonymous reviewers, who
  noted that our proof strongly relied on the use of \cite[Theorem 4.1]{ben-hof-sor-05}, which turns out to be incorrect. A detailed account on this issue is provided in Appendix~\ref{app:ben}. The same mistake
  occurred  in \cite{schechtman_phd} which incorrectly states a nonconvergence result toward sharply repulsive points.
   Thus, the question of the generic convergence of SGD toward local minimizers remains open in the non-weakly convex setting.
\end{remark}

\input{th3}

\input{proof_duflo}
 \bibliographystyle{plain}
 \bibliography{math}
\input{app}
\input{ben_APT}
\end{document}

%% file: th3.tex
\section{Proof of Theorem~\ref{th:avoid_trap}}\label{sec:proof_main}

\label{sec:proofs}
From now on, we assume without restriction that $x^*=0$. Thus,
$\nabla_{\cM}f(0) = 0$, and there exists a vector $w \in \cT_{0} \cM$ such
that $\scalarp{w}{\cH_{f,\cM}(0)(w)}< 0$.
The general idea of the proof of Theorem~\ref{th:avoid_trap} is that on the
event $[x_n\to 0]$, the function $P_\cM$ is defined for all large $n$,
enabling us to write $x_n = y_n + z_n$ for these $n$, where $y_n = P_\cM(x_n)$.
The iterates $(y_n)$ can then be written under the form of a standard
smooth \emph{Robbins-Monro algorithm} for which the trap avoidance can be
established by the technique of Brandi\`ere and Duflo \cite{bra-duf-96}.
In this setting, the remainders $z_n$ will be shown to be small enough
so as not to alter fundamentally the approach of \cite{bra-duf-96}.

\textbf{Outline of the proof.}
\begin{itemize}

  \item \textit{Strengthened avoidance of traps.} First, in Section~\ref{sec:prel_traps} we provide a strengthened version of the avoidance of traps result of \cite{bra-duf-96, pem-90, tarrespieges}. In a first approximation, this result, given in Proposition~\ref{prop:duflo}, states that if an $\bbR^d$-valued sequence $(y_n)$ satisfying a recursion of the form
  \begin{equation}\label{eq:duflo_outline}
     y_{n+1} = y_{n} - \gamma_n D( y_n)  + \gamma_n \tilde{\eta}_{n+1}
    + \gamma_n \varrho_{n+1} + \gamma_n \tilde{\varrho}_{n+1}  \, ,
  \end{equation}
where $D : \bbR^d \rightarrow \bbR^d$ is a function such that $J_D(0)$ has at
least one eigenvalue with a negative real part, the noise $\tilde{\eta}_{n+1}$
is omnidirectional in some sense, and the terms $\varrho_{n+1}$ and
$\tilde\varrho_{n+1}$ are small perturbation terms (see
Proposition~\ref{prop:duflo} for the exact assumptions), then $\bbP(y_n
\rightarrow 0) = 0$. The only difference with the assumptions of
\cite{bra-duf-96, pem-90, tarrespieges} is in the presence of the term
$\tilde{\varrho}_{n+1}$.  To not interrupt the exposition we provide a complete
proof of Proposition~\ref{prop:duflo} in Section~\ref{sec:proof_duf}. In the
context of our algorithm and leaving technical details aside, denoting $y_n =
P_{\cM}(x_n)$, Equation~\eqref{eq:duflo_outline} appears from a Taylor
expansion of $P_{\cM}$, with $D(y_n) = \nabla_\cM f(y_n)$. Here,
$\tilde{\varrho}_{n+1}$ appears from the Verdier condition and its norm will be
controlled by $\dist(x_n, \cM)$.

  \item \textit{Construction of $(y_n)$.} The goal of Section~\ref{sec:app_sgd}
is to construct a sequence satisfying Equation~\eqref{eq:duflo_outline}. A
natural attempt is to define $(y_n) := (P_{\cM}(x_n))$. However, since
$P_{\cM}$ is defined only on a neighborhood of $\cM$ (say $U$), in our
construction we fix $N \in \bbN$ and for $n \geq N$, define $y_n^{N}$ as
$P_{\cM}(x_n)$ if the iterates $\{x_N, x_{N+1}, \dots, x_n\} \subset U$ and
$(y_n^{N})$ satisfies (by construction) Equation~\eqref{eq:duflo_outline}
otherwise.  In this context, $D$ becomes an appropriate extension of
$\nabla_{\cM}f$ outside $\cM$. As explained in Section~\ref{sec:app_sgd}, if
we establish $\bbP(y_n^N \rightarrow 0) = 0$, for all $N\in \bbN$, then this
will imply that $\bbP(x_n \rightarrow 0) = 0$. Thus, the goal of the next
sections will be to verify the assumptions of Proposition~\ref{prop:duflo} for
the sequence $(y_n^N)$.

  \item \textit{The sequence $(y_n^N)$ can be written as a Robbins-Monro
algorithm.} In Section~\ref{sec:app_sgd} we show that the constructed sequence
$(y_n^N)$ satisfies indeed Equation~\eqref{eq:duflo_outline}, and verifies almost all the assumptions of the general Proposition~\ref{prop:duflo}. This will be the content of Proposition~\ref{yduf}. In short, it follows from a Taylor expansion of $P_{\cM}$.
 Here, we put a special
emphasis on the term $\tilde{\varrho}_{n}^N$ (analogous to
$\tilde{\varrho}_{n}$ above), which will be shown to satisfy
$\bbE[\norm{\tilde{\varrho}_{n+1}^N}|\mcF_n] = \mathcal O (\dist(x_n,\cM))$. In the remaining sections our goal will be to control the convergence rate of $(\varrho_{n}^N)$ toward zero, which will allow us to apply Proposition~\ref{prop:duflo}.

\item \textit{Convergence rates of $(\varrho_n^N)$.} In Section~\ref{sec:drift} we establish the key drift inequality on $\dist(x_n, \cM)$, which allows us in Sections~\ref{sec:conv_man}--\ref{sec:contr_weight} establish that
\[
  \limsup \chi_n^{-1/2} \sum_{i=n}^{\infty}\gamma_i\bbE \left[\dist(x_i, \cM) |\mcF_n \right]= 0 \, ,
\]
where $\chi_n = \sum_{i=n}^{\infty} \gamma_i^2$.

\item \textit{End of proof.} Finally, in Section~\ref{sec:avoid_trap_end} we finish the proof of Theorem~\ref{th:avoid_trap} by applying Proposition~\ref{prop:duflo} to the sequence $(y_n^N)$ and precisely invoking the intermediate results allowing us to verify its assumptions.
\end{itemize}

\subsection{Preliminary:  Avoidance of Traps in the Smooth Case}\label{sec:prel_traps}

The following proposition is an avoidance of traps results in the smooth
case, which is tailored to our assumptions on the projected iterates on the
manifold $M$.  Its proof combines ideas from Brandi\`ere and Duflo
\cite{bra-duf-96} on the one hand, and from Tarr\`es \cite{tarrespieges} on the
other.  Our proposition might have some interest of its own, since it is
more general than the corresponding result in \cite{bra-duf-96}. We also note
that its proof corrects some errors found in the latter paper\footnote{We would
like to thank an anonymous reviewer for pointing out this fact.}.

Given a positive integer $d$, the statement of the proposition makes use of a
field $D : \bbR^d \rightarrow \bbR^d$.
We shall assume that $D$ is defined
on $\bbR^d$ and is $C^3$ in some neighborhood of $0$, with $D(0) = 0$.  We
shall also assume that $J_{D}(0)$, the Jacobian of $D$ at $0$, has at least one
eigenvalue with a strictly negative real part. Denote as $d^-$ the dimension of
the invariant subspace of $J_{D}(0)$ associated with these eigenvalues, and let
$d^+ = d - d^-$. We shall need the spectral factorization
\begin{equation*}
  J_{D}(0) = P \begin{pmatrix}
    J^+ & 0 \\
    0 & J^-
\end{pmatrix} P^{-1}\, ,
\end{equation*}
of $J_{D}(0)$, where $J^{-} \in \bbR^{d^- \times d^{-}}$ contains those Jordan
blocks of $J_{D}(0)$ that are associated with the eigenvalues with negative
real parts.

We now state our proposition. In all the remainder, given a filtration
$(\mcF_n)$ in a probability space $(\Omega, \mcF, \bbP)$, we shall write
$\bbE_n[\cdot] = \bbE[\cdot | \mcF_n]$ and $\bbP_n(\cdot) = \bbP()\cdot |
\mcF_n)$.

\begin{proposition}
\label{prop:duflo}
Let $(\Omega, \mcF, \bbP)$ be a probability space, $(\mcF_n)$ a filtration and
$(\gamma_n)$ a sequence of deterministic nonnegative step sizes such that
$\sum_{i=0}^{\infty} \gamma_i = +\infty$ and $\sum_{i=0}^{\infty} \gamma_i^2 < +\infty$.

Consider the $\bbR^{d}$--valued stochastic process $(y_n)$ given by
\begin{equation}\label{eq:duf}
   y_{n+1} = y_{n} - \gamma_n D( y_n)  + \gamma_n \tilde{\eta}_{n+1}
  + \gamma_n \varrho_{n+1} + \gamma_n \tilde{\varrho}_{n+1}  \, ,
\end{equation}
where $y_0$ is $\mcF_0$-measurable, the map $D$ is as above, the sequences
$(\tilde{\eta}_n), (\varrho_n)$, and $(\tilde{\varrho}_n)$ are
$(\mcF_n)$-adapted, and both $\norm{\tilde{\eta}_n}$ and
$\norm{\tilde{\varrho}_n}$ have finite fourth moments for each $n$. Write
$(\tilde\eta^+_n, \tilde{\eta}^{-}_n) = P^{-1} \tilde\eta_n$, where
$\tilde\eta^\pm_n \in \bbR^{d^\pm}$. Let $\Gamma \in \mcF$ be an arbitrary event and assume that on $[y_n \rightarrow 0] \cap \Gamma$ the following almost surely holds.
 \begin{enumerate}[{\it i)}]
 \item\label{Eeta0} $\forall n \in \bbN$, $\bbE_n\tilde{\eta}_{n+1} = 0$.
 \item\label{eta-4mom}
 $\limsup\bbE_n\norm{\tilde{\eta}_{n+1}}^4 < + \infty$.
\item \label{eq:unstable_exit}
$\liminf \bbE_n\norm{\tilde{\eta}^{-}_{n+1}} > 0 \, .$
\item\label{sum-rho} $\sum_{i=0}^{\infty} \norm{\varrho_{i+1}}^2 < + \infty \, .$
\item\label{rho-4mom}
 $\limsup \bbE_n \norm{\tilde{\varrho}_{n+1}}^4  < + \infty$.
\item\label{rho-2mom} $\lim \bbE_n \norm{\tilde{\varrho}_{n+1}}^2 = 0$.

\item\label{hyp:dufl_sumtilde}
  $\displaystyle{\limsup \chi_n^{-1/2}
  \sum_{i=n}^{\infty} \gamma_i \bbE_n \norm{\tilde{\varrho}_{i+1}}
      = 0}$, where $\chi_n = \sum_{i=n}^{\infty}\gamma_i^2$.
\end{enumerate}
Then, $\bbP(\Gamma \cap  [y_n \rightarrow 0]) = 0$.
\end{proposition}
This proposition is similar to \cite[Theorem 1]{bra-duf-96}, except for the
presence of the sequence $(\tilde{\varrho}_{n})$. To not interrupt the exposition we present its proof in Section~\ref{sec:proof_duf}.

\subsection{Application to Algorithm~\eqref{eq:sgd}}\label{sec:app_sgd}

To apply the results of the preceding section we need, first, to find a
candidate for $D$, this is the purpose of the next lemma. Its proof readily
follows from results of Section~\ref{sec-prel}.

\begin{lemma}\label{lm:fopim}
  Let Assumption~\ref{hyp:activ_strict} hold and let $r>0$ be such that $P_{\cM} : B(0,r) \rightarrow \cM$ is well defined and is $C^4$ and that there is a $C^5$ function $F: B(0,r) \rightarrow \bbR$ that agrees with $f$ on $\cM \cap B(0,r)$.
  Then, the function $F \circ P_{\cM}$ is $C^4$ on $B(0,r)$ and for $y \in \cM \cap B(0,r)$, we have:
  \begin{equation*}
    \nabla (F \circ P_{\cM})(y) = \nabla_{\cM}f(y) \, .
  \end{equation*}
  Moreover, for $w \in \cT_{\cM}0$:
  \begin{equation*}
    \scalarp{w}{\cH_{f, \cM}(0)(w)} =\scalarp{w}{\cH_{F \circ P_{\cM}}(0)w}\, ,
  \end{equation*}
  where $\cH_{F \circ P_{\cM}}(0)$ is the usual Euclidean Hessian of $F \circ P_{\cM}$ at $0$.
\end{lemma}

Let $r_1>0$ be chosen in a way that the conditions of
Definition~\ref{def:ver_ang} and Lemma~\ref{lm:fopim} are satisfied on $\overline{B(0,r_1)}$. In the remainder, we fix an $r >0$ such that $0<r < r_1$.
The value of $r$, while always satisfying this requirement, will be adjusted
in the course of the proof.

First, by Tietze's extension theorem the function $\nabla (F \circ P_{\cM}) : B(0,r)
\rightarrow \bbR^d$ can be extended to a bounded continuous function $D :
\bbR^d \rightarrow \bbR^d$ that we shall use in the remainder of the paper.

Second, to reduce technical issues, we notice that as in \cite[Section I.2]{bra-duf-96} (see also Appendix~\ref{app:coupl}) to prove Theorem~\ref{th:avoid_trap}
we can actually replace Assumption~\ref{hyp:noise_exit} by the following, more
easy to handle, assumption.
\begin{assumption}\label{hyp:noise_exit_strong}
  Almost surely, the sequence $(\eta_n)$ is such that $\bbE_n[\eta_{n+1}] = 0$ and there is $A, B >0$ such that for all $n \in \bbN$, we have:
  \begin{equation*}
    \bbE_n\norm{\eta_{n+1}}^4 \leq B \, , \quad   \bbE_n\norm{\eta_{n+1}^{-}} \geq A \, .
  \end{equation*}
\end{assumption}

Given an integer $N \geq 0$, we define the probability event
\[
\cA_N = \left[ \forall n \geq N, \ \norm{x_n} \leq r  \right] .
\]
Note that the sequence of events $(\cA_N)$ is increasing for the inclusion.
Furthermore, it holds that
\[
\left[ x_n \to 0 \right] \subset \bigcup_{N=0}^\infty \cA_N =
 \lim_{N\to\infty} \cA_N.
\]
Thus,
\[
\bbP \left[ x_n \to 0 \right] =
\bbP \left[ \left[ x_n \to 0 \right] \cap \lim \cA_N \right] =
\lim_{N\to\infty} \bbP \left[ \left[ x_n \to 0 \right] \cap \cA_N \right].
\]
Consequently, given an arbitrary $\delta > 0$, there is an integer
$N(\delta) \geq 0$ such that
\begin{equation}
\label{cvg-delta}
\bbP \left[ \left[ x_n \to 0 \right] \cap \cA_{N(\delta)} \right] \geq
\bbP \left[ x_n \to 0 \right] - \delta .
\end{equation}
For an integer $N \geq 0$, define the stopping time
\[
\tau_N = \inf \{ n \geq N, \ \norm{x_n} > r \} ,
\]
with $\inf\emptyset = \infty$,
and recall from the definition of $r$ that for $N\leq n < \tau_N$, the
projection $P_\cM(x_n)$ is well-defined. Define recursively the process
$(y^N_n)_{n\geq N-1}$ as follows: $y^N_{N-1} = 0$,
\[
y^N_n = \begin{cases}
P_\cM(x_n) \ \text{if} \ N \leq n < \tau_N, \\
y^N_{n-1} - \gamma_{n-1} D(y^N_{n-1})
   + \gamma_{n-1} J_{P_\cM}(y^N_{n-1}) \eta_{n}
   \ \text{if} \ n = \tau_N, \\
 y^N_{n-1} - \gamma_{n-1} D(y^N_{n-1}) + \gamma_{n-1} \eta_{n},
 \ \text{ otherwise},
\end{cases}
\]
and let
\[
z^N_n = (x_n - y^N_n) \1_{n < \tau_N} \quad \text{for} \ n \geq N .
\]
Observe that $y^N_n$ and $z^N_n$ are both $\mcF_n$--measurable for all
$n\geq N$. To establish Theorem~\ref{th:avoid_trap}, we shall show that for
each $N \geq 0$,
\begin{equation}
\label{cpl}
\bbP\left[ y^N_n \xrightarrow[n\to\infty]{} 0 \right] = 0.
\end{equation}
Indeed, on the event $\cA_{N(\delta)}$, it holds that
$y^{N(\delta)}_n = P_\cM(x_n)$ for $n\geq N(\delta)$, thus,
\[
\left[ \left[ x_n \to 0 \right] \cap \cA_{N(\delta)} \right]
 \subset
\left[ \left[ y^{N(\delta)}_n \to 0 \right] \cap \cA_{N(\delta)} \right] .
\]
Consequently, with the convergence~\eqref{cpl} at hand, we get from
Inequality~\eqref{cvg-delta} that $\bbP() x_n \to 0) \leq \delta$. Since $\delta$ is
arbitrary, we obtain that $\bbP( x_n \to 0) = 0$.

In the remainder of this section, $N\geq 0$ is a fixed integer.

\subsubsection{Projected Iterates as a Robbins-Monro Algorithm}\label{sec:proj_iter}

The next
proposition shows that the sequence $(y_n^N)_{n\geq N}$ satisfies a recursion
of the form~\eqref{eq:duf}, satisfying almost all the assumptions of Proposition~\ref{prop:duflo}.

\begin{proposition}
\label{yduf}
Let Assumptions~\ref{hyp:model}--\ref{hyp:activ_strict} and
\ref{hyp:noise_exit_strong} hold. Then, the sequence $(y^N_n)_{n\geq N}$
satisfies the recursion:
\[
  y^N_{n+1} = y^N_n - \gamma_n D(y^N_n) + \gamma_{n}\tilde{\eta}^N_{n+1}
  + \gamma_n \varrho^N_{n+1} + \gamma_n \tilde\varrho^N_{n+1}\, ,
\]
where the random sequences $(\tilde{\eta}^N_{n})_{n\geq N}$,
$(\varrho^N_n)_{n\geq N}$, and $(\tilde\varrho^N_n)_{n\geq N}$ are adapted to
$(\mcF_n)$. Moreover, there is $C >0$ such that for all $n \geq N$,
\begin{enumerate}[i)]
  \item\label{bnd-rho}
  $\norm{\varrho^N_{n+1}} \leq C \gamma_n
     ( 1 + \norm{\eta_{n+1}}^2 ) \1_{\tau_N > n+1}$.
  \item\label{bnd-zeta}
   $\norm{\tilde\varrho^N_{n+1}} \leq C \norm{z^N_n} (1+ \norm{\eta_{n+1}})$.
  \item \label{teta-noise}
   $\bbE_n \tilde{\eta}^N_{n+1} = 0$, and
   $\bbE_n\norm{\tilde{\eta}^N_{n+1}}^4 < C$.
\setcounter{compteur-enumeration}{\value{enumi}}
\end{enumerate}
We furthermore have:
\begin{enumerate}[i)]
\setcounter{enumi}{\value{compteur-enumeration}}
\item\label{E-=L-}
The subspace $E^-$ defined before Assumption~\ref{hyp:noise_exit} is included in $T_{0} \cM$ and $E^-_D$, the eigenspace of the matrix $J_D(0)$ corresponding to its
negative eigenvalues.
\item\label{eta-omni}
 On the event $[y^N_n \rightarrow 0]$, it holds that
 \begin{equation*}
   \liminf_{n} \bbE_n\norm{P_{E^-_D} \tilde \eta^{N}_{n+1}} \geq \liminf_n \bbE_n \norm{P_{E^-} \tilde\eta^N_{n+1}} > 0\, .
 \end{equation*}
\end{enumerate}
\end{proposition}

To prove Proposition~\ref{yduf} we will need the following, technical lemma.

\begin{lemma}
\label{lm:proj_taylor}
For $r$ small enough, there is $C > 0$ such that for $x,x' \in B(0,r)$, we
have:
   \begin{equation*}
     y' - y =  J_{P_\cM}(y) (x' - x) + R_1(x,x',y) + R_2(x, x') \, ,
   \end{equation*}
   where $y', y = P_{\cM}(x'), P_{\cM}(x)$, and where
   $\norm{R_1(x, x',y)} \leq C \norm{x'-x} \norm{x-y}$,
   and $\norm{R_2(x,x')} \leq C \norm{x' - x}^2$.
 \end{lemma}
 \begin{proof}
Since $P_{\cM}$ is $C^2$ near zero, there is $\varepsilon >0$ such that
$t \mapsto P_{\cM}( x + t(x' - x))$ is $C^2$ on
$(- \varepsilon, 1 + \varepsilon)$. Hence, by Taylor's theorem, we have
   \begin{equation*}
     y' - y = J_{P_\cM}(x)(x' -x ) + R_2(x', x) \, ,
   \end{equation*}
with $\norm{R_2(x',x)} \leq C \norm{x' - x}^2$, where $C$ is a bound on the second derivatives of $P_M$. Similarly,
since $P_{\cM}$ is $C^2$, $x \mapsto J_{P_\cM}(x)$ is Lipschitz continuous.
Therefore, for some $C >0$, $\norm{J_{P_\cM}(x)-J_{P_\cM}(y)} \leq
C \norm{x- y}$, which finishes the proof.
 \end{proof}

\begin{proof}[Proof of Proposition~\ref{yduf}]
Letting $n \geq N$, we write
\[
y^N_{n+1} = P_\cM(x_{n+1}) \1_{\tau_N > n+1} +
 \left(y^N_{n} - \gamma_{n} D(y^N_{n})\right) \1_{\tau_N \leq n+1} +
 \gamma_n\left( J_{P_\cM}({y^N_n}) \1_{\tau_N = n+1} + \1_{\tau_N\leq n}\eta_{n+1} \right)
   ,
\]
accepting the small notational abuse in the expression $P_\cM(x_{n+1})
\1_{\tau_N > n+1}$, since the projection might not be defined when the
indicator is zero. Similar abuses will also be made in the derivations below.

Using Lemma~\ref{lm:proj_taylor} and Equation~\eqref{eq:sgd}, we obtain
\begin{align*}
y^N_{n+1}
&= \left( y^N_n + J_{P_\cM}(y^N_n)( x_{n+1} - x_n) \right) \1_{\tau_N > n+1}
  + \gamma_n \varrho^N_{n+1}  + \gamma_n \zeta^N_{n+1} \\
&\phantom{=} +
 \left(y^N_{n} - \gamma_{n} D(y^N_{n}) \right) \1_{\tau_N \leq n+1}
 + \gamma_n\left( J_{P_\cM}(y^N_n) \1_{\tau_N = n+1} + \1_{\tau_N\leq n}\eta_{n+1} \right)
    \\
&= \left( y^N_n - \gamma_n J_{P_\cM}(y^N_n)v_n + \gamma_n J_{P_\cM}(y^N_n) \eta_{n+1} \right)
    \1_{\tau_N > n+1}  + \gamma_n \varrho^N_{n+1}  + \gamma_n \zeta^N_{n+1} \\
&\phantom{=} +
 \left(y^N_{n} - \gamma_{n} D(y^N_{n}) \right) \1_{\tau_N \leq n+1}
 + \gamma_n\left( J_{P_\cM}(y^N_n) \1_{\tau_N = n+1} + \1_{\tau_N\leq n}\eta_{n+1} \right)
   ,
\end{align*}
where $\varrho^N_{n+1}$ and $\zeta^N_{n+1}$ are $\mcF_{n+1}$--measurable, and
satisfy with the notations of Lemma~\ref{lm:proj_taylor}
\[
\norm{\zeta^N_{n+1}} = \gamma_n^{-1} \norm{R_1(x_n, x_{n+1}, y^N_n)}
 \1_{\tau_N > n+1}
 \leq C \gamma_n^{-1} \norm{x_{n+1} - x_n} \norm{z^N_n}
 \leq C (1 + \norm{\eta_{n+1}}) \norm{z^N_n}
\]
(in the last inequality, we used that $\norm{v_n}$ is bounded on
$[\tau_N > n]$), and
\begin{align*}
\norm{\varrho^N_{n+1}}
 &= \gamma_n^{-1} \norm{R_2(x_n, x_{n+1})} \1_{\tau_N > n+1} \\
 &\leq C \gamma_n^{-1} \norm{x_{n+1} - x_n}^2 \1_{\tau_N > n+1} \\
 &\leq C \gamma_n ( 1 + \norm{\eta_{n+1}}^2 ) \1_{\tau_N > n+1}.
\end{align*}
Using Lemma~\ref{lm:proj} in conjunction with the Verdier condition
\eqref{verd} of Definition~\ref{def:ver_ang}, we also have
\[
J_{P_\cM}(y^N_n)v_n \1_{\tau_N > n+1} = P_{\cT_{y^N_n} \cM}(v_n) \1_{\tau_N > n+1}
  = \nabla_\cM f(y^N_n) \1_{\tau_N > n+1} + \tilde\zeta^N_{n+1}
  = D(y^N_n) \1_{\tau_N > n+1} + \tilde\zeta^N_{n+1},
\]
where $\tilde\zeta^N_{n+1}$ is $\mcF_{n+1}$--measurable, and satisfies
\[
\norm{\tilde\zeta^N_{n+1}} \leq C \norm{x_n - y^N_n} \1_{\tau_N > n+1} \leq
 C \norm{z^N_n} .
\]
Gathering these expressions, we get
\[
y^N_{n+1}
=  y^N_n - \gamma_n D(y^N_{n}) + \gamma_n \tilde\eta^N_{n+1}
 + \gamma_n \varrho_{n+1} + \gamma_n \tilde\varrho_{n+1},
\]
where
\begin{align}
\tilde\eta_{n+1}^N &= \left( \1_{\tau_N > n} J_{P_\cM}(y^N_n)
 + \1_{\tau_N\leq n} \right) \eta_{n+1}, \ \text{and}
\label{teta} \\
\tilde\varrho^N_{n+1} &=
  \zeta^N_{n+1} + \tilde\zeta^N_{n} \nonumber .
\end{align}
The assertions~\ref{bnd-rho}) and~\ref{bnd-zeta}) of the statement are
obtained from what precedes.

The noise $\tilde\eta_{n}^N$ is obviously $\mcF_n$--measurable. Moreover,
$\bbE_n \tilde\eta_{n+1}^N = 0$ since $\1_{\tau_N > n} J_{P_\cM}(y^N_n) +
\1_{\tau_N\leq n}$ is $\mcF_n$--measurable. The last bound in~\ref{teta-noise})
follows from Assumption~\ref{hyp:noise_exit_strong}.

Assertion~\ref{E-=L-}) follows from Lemma~\ref{lm:fopim}.

To establish \ref{eta-omni}), we notice that, since $E^{-} \subset E_D^{-}$,
\begin{align*}
\norm{P_{E_D^-} \tilde \eta^N_{n+1}} \geq \norm{ P_{E^-}\tilde\eta^N_{n+1}}
 &= \norm{P_{E^{-}} J_{P_\cM}(y^N_n)\eta_{n+1}} \1_{\tau_N > n} +
 \norm{P_{E^{-}} \eta_{n+1}} \1_{\tau_N\leq n} \\
&\geq \norm{P_{E^{-}}\eta_{n+1}}
  - \norm{P_{E^{-}} J_{P_\cM}(y_n^N) \eta_{n+1} - P_{E^-} \eta_{n+1}}
  \1_{\tau_N > n}.
\end{align*}
On the event $[y_n^N \rightarrow 0]$, it holds that $J_{P_\cM}(y^N_n) \rightarrow
J_0$. By Lemma~\ref{lm:proj}, $J_{0}$ is the orthogonal projection on
$\cT_{0} \cM$, thus, $\lim_{y^N_n \rightarrow 0} P_{E^{-}} J_{P_\cM}(y^N_n) =
P_{E^{-}}$. Consequently, we obtain on the event $[y^N_n \rightarrow 0]$:
\begin{equation*}
  \begin{split}
    \liminf_n \bbE_n\norm{\norm{P_{E_D^-} \tilde \eta^N_{n+1}}} &\geq  \liminf_n \bbE_n \norm{ P_{E^-}\tilde\eta^N_{n+1}}\\
    &\geq
    \liminf_n \bbE_n\norm{\eta_{n+1}^{-}}
    - \limsup_n \left(\norm{P_{E^{-}} J_{P_\cM}(y^N_n) - P_{E^{-}}} \,
   \bbE_n \norm{\eta_{n+1}} \right)\\
    &\geq \liminf_n \bbE_n\norm{\eta_{n+1}^{-}}\, ,
  \end{split}
\end{equation*}
and by Assumption~\ref{hyp:noise_exit_strong}. Proposition~\ref{yduf} is proven.
\end{proof}

\subsubsection{Drift Inequality}\label{sec:drift}

Proposition~\ref{yduf} shows that we may apply Proposition~\ref{prop:duflo} to the sequence $(y_n^N)$ as soon as we are able to control the rate of convergence of $(z_n^N)$ toward zero. The latter will be the goal of this and subsequent sections.

The next lemma establishes the key inequality in our proof. Recall that
$\beta>0$ is the one given by the angle condition of
Definition~\ref{def:ver_ang}.
\begin{lemma}
\label{lm:En-z}
Under the assumptions of Proposition~\ref{zsympa}, there is $C >0$ such that
if $r >0$ is chosen small enough, then
\[
\bbE_n\norm{z^N_{n+1}}^2 \leq
   \norm{z^N_n}^2 - \gamma_n \beta\norm{z^N_n} + C \gamma_n^2 \, .
\]
\end{lemma}
\begin{proof}
We shall use the notation
\[
p_n^N = x_n - y^N_n,
\]
which enables us to write $z_n^N = p_n^N \1_{n < \tau_N}$.

We start with the development
\begin{align}
\norm{z^N_{n+1}}^2 &= \norm{p^N_{n+1}}^2 \1_{n+1 < \tau_N} \nonumber \\
 &\leq \norm{p^N_{n+1}}^2 \1_{n < \tau_N} =
  \norm{p^N_{n+1} - p^N_n + p^N_n}^2 \1_{n < \tau_N} \nonumber \\
 &= \norm{z_n^N}^2 +
  2 \scalarp{x_{n+1} - x_n}{z^N_n} - 2 \scalarp{y^N_{n+1} - y^N_n}{z^N_n}
  + \norm{p^N_{n+1} - p^N_n}^2 \1_{n < \tau_N} .
\label{iter-z}
\end{align}
We now deal separately with each of the three rightmost terms in the last
expression.

We first show that
\begin{equation}
\label{lm:<y,z>}
\bbE_n| \scalarp{y^N_{n+1} - y^N_n}{z^N_n} | \leq
  C \gamma_n\norm{z_n^N}^2 + C\gamma_n^2.
\end{equation}
By Proposition~\ref{yduf},
\[
\scalarp{y^N_{n+1} - y^N_n}{z^N_n} =
\gamma_n \scalarp{- D(y^N_{n}) + \tilde\eta^N_{n+1}
 + \varrho^N_{n+1} + \tilde\varrho^N_{n+1}}{z^N_n} .
\]
We have $\scalarp{D(y^N_{n})}{z^N_n} = \scalarp{\nabla_\cM f(y^N_{n})}{z^N_n}
 = 0$ since $\nabla_\cM f(y^N_{n}) \in \cT_{y^N_n} \cM$. Furthermore,
we get from Equation~\eqref{teta} that
\[
\1_{n < \tau_N} \tilde\eta^N_{n+1}  =
\1_{n < \tau_N} J_{P_\cM}(y^N_n) \eta_{n+1}
= \1_{n < \tau_N} P_{\cT_{y^N_n} \cM}(\eta_{n+1})
\]
by Lemma~\ref{lm:proj}, thus, $\scalarp{\tilde\eta^N_{n+1}}{z^N_n} = 0$.
As a consequence,
\[
2| \scalarp{y^N_{n+1} - y^N_n}{z^N_n} | \leq
 \gamma_n ( \norm{z^N_n}^2 +
  \norm{\varrho^N_{n+1} + \tilde\varrho^N_{n+1}}^2 )
 \leq \gamma_n \norm{z^N_n}^2 +
  2 \gamma_n (\norm{\varrho^N_{n+1}}^2 + \norm{\tilde\varrho^N_{n+1}}^2).
\]
From Proposition~\ref{yduf} again, we have
\[
 \bbE_n\norm{\varrho^N_{n+1}}^2 \leq
 C \gamma_n^2 \bbE_n( 1 + \norm{\eta_{n+1}}^4 ) \1_{\tau_N > n+1} \leq
 C \gamma_n^2 \bbE_n( 1 + \norm{\eta_{n+1}}^4 )
  \leq C \gamma_n \, ,
\]
  and
\[
 \bbE_n\norm{\tilde\varrho^N_{n+1}}^2
  \leq C \norm{z^N_n}^2 (1 + \bbE_n\norm{\eta_{n+1}}^2) \leq
  C \norm{z^N_n}^2 .
\]
Inequality~\eqref{lm:<y,z>} is obtained by combining these inequalities.

We next show succinctly that
\begin{equation}
\label{lm:p-p}
 \bbE_n \norm{p^N_{n+1} - p^N_n}^2 \1_{n < \tau_N} \leq C \gamma_n^2 .
\end{equation}
Indeed,
\begin{multline*}
 \norm{p^N_{n+1} - p^N_n}^2 \1_{n < \tau_N} =
 \norm{x_{n+1} - x_n - (y^N_{n+1} - y^N_n)}^2 \1_{n < \tau_N} \\
  \leq C \gamma_n^2 \left(\norm{v_n}^2 + \norm{\eta_{n+1}}^2 +
  \norm{D(y^N_n)}^2 + \norm{\tilde{\eta}^N_{n+1}}^2 +
   \norm{\varrho^N_{n+1}}^2 + \norm{\tilde\varrho^N_{n+1}}^2\right)
  \1_{n < \tau_N},
\end{multline*}
and the result follows by standard calculations making use of the
results of Proposition~\ref{yduf}.

We finally deal with the term $\scalarp{x_{n+1} - x_n}{z^N_n}$.  Since $\bbE_n
\eta_{n+1} = 0$, we have $\bbE_n \scalarp{x_{n+1} - x_n} {z^N_n} = - \gamma_n
\scalarp{v_n}{z^N_n}$. Since the angle condition is satisfied on $B(0,r)$, and $\norm{x_n} \leq r$ when $z^N_n \neq
0$, we obtain that
\[
  \bbE_n \scalarp{x_{n+1} - x_n}{z^N_n} \leq
- \gamma_n \beta \norm{z^N_n} .
\]
Getting back to Inequality~\eqref{iter-z}, and using this result in conjunction
with the inequalities~\eqref{lm:<y,z>} and~\eqref{lm:p-p}, we obtain that
\begin{equation*}
  \begin{split}
    \bbE_n\norm{z^N_{n+1}}^2 &\leq \norm{z^N_n}^2 + C\gamma_n \norm{z^N_n}^2
      - 2 \gamma_n \beta \norm{z^N_n} + C \gamma_n^2  \\
      &\leq \norm{z_n^{N}}^2 + \gamma_n \norm{z_n^N}(C r - 2 \beta) + C \gamma_n^2 \, ,
  \end{split}
\end{equation*}
where in the last inequality we have used the fact that $\norm{x_n} \leq r$ on the event $[n < \tau_N]$. Choosing $r$ small enough to satisfy $Cr \leq \beta$, completes the proof.
\end{proof}

\subsubsection{Convergence Rates to the Manifold}\label{sec:conv_man}
Using Lemma~\ref{lm:En-z} we are now able to obtain first convergence rates for $(z_n^N)$.
We recall that the exponent $\alpha$ defined in
Assumption~\ref{hyp:model}-\ref{hyp:steps}) is in the interval $(1/2,1]$.

\begin{lemma}\label{lm:zn2_as_bound}
Under the assumptions of Proposition~\ref{zsympa}, for each
$a \in (0, 2\alpha - 1)$, it holds that $n^{a} \norm{z_n^N}^2 \to_n 0$ almost
surely.
\end{lemma}
\begin{proof}
Let $a \in (0, 2\alpha - 1)$ be arbitrary.
For $n \geq N$, denote $u_n = n^{a} \norm{z_n^N}^2$. From Lemma~\ref{lm:En-z}
and Assumption~\ref{hyp:model}-\ref{hyp:steps}), we obtain that
  \begin{equation*}
    \bbE_n u_{n+1} \leq \left(\frac{n+1}{n}\right)^{a} u_n - \beta c_1 \frac{(n+1)^a}{n^{\alpha}} \norm{z_n^N} + Cc_2\frac{(n+1)^{a}}{n^{2\alpha}}\, .
  \end{equation*}
  Noticing that there is $C'>0$ such that for all $n >1$, $(1+1/n)^{a} \leq 1 + C'/n$, we obtain:
  \begin{equation*}
    \begin{split}
          \bbE_n u_{n+1} &\leq u_n + C' n^{-1} u_n - c_1\beta n^{a - \alpha}\norm{z_n^N}  + 2^aC c_2 n^{a - 2\alpha} \\
          &\leq u_n + C'n^{a -1} \norm{z_n^N}^2 - c_1\beta n^{a - \alpha}\norm{z_n^N}  + 2^aC c_2 n^{a - 2\alpha}\, .
    \end{split}
  \end{equation*}
  In particular, since $\alpha < 1$, if $r>0$ is chosen small enough, then $c_1\beta n^{a - \alpha}\norm{z_n^N} \geq C'n^{a -1} \norm{z_n^N}^2$ and we obtain:
  \begin{equation*}
    \bbE_n u_{n+1} \leq u_n + \frac{2^aC c_2}{n^{2\alpha- a}} \, .
  \end{equation*}
  Since $2 \alpha - a>1$, we obtain by Robbins-Siegmund's theorem
 \cite{robbins1971convergence} that $n^a \norm{z_n^N}^2$ converges almost
surely as $n\to\infty$. Since $a$ is arbitrary in $(0, 2\alpha - 1)$, this
limit is zero.
\end{proof}

\subsubsection{Control of the Weighted Sums}\label{sec:contr_weight}

Finally, we can now control the weighted sum of the sequence $(z_n^N)$. This, as it will be clear from Section~\ref{sec:avoid_trap_end}, was the only remaining point to apply Proposition~\ref{prop:duflo} to the sequence $(y_n^N)$.
\begin{proposition}
\label{zsympa}
Let Assumptions~\ref{hyp:model}--\ref{hyp:activ_strict} and
\ref{hyp:noise_exit_strong} hold true. Then, the sequence $(z_n^N)$
satisfies
\[
\chi_n^{-1/2} \sum_{i=n}^{\infty} \gamma_i\bbE_n\norm{z_i^N}
 \xrightarrow[n\to\infty]{\text{a.s.}} 0 .
\]
\end{proposition}
\begin{proof}
Let $C > 0$ be the constant provided in the statement of
Proposition~\ref{lm:En-z}. For each $n \geq N$ and $k \geq n$, we obtain from
this proposition that
\[
\bbE_n \norm{z^N_{k+1}}^2 \leq \norm{z^N_{n}}^2 - \beta
 \sum_{i=n}^k \gamma_i \bbE_n \norm{z^N_{i}} + C \sum_{i=n}^k \gamma_i^2 .
\]
Taking $k$ to infinity, we get that
\[
\beta \sum_{i=n}^{\infty}\gamma_i \bbE_n\norm{z^N_{i}}
  \leq \norm{z^N_{n}}^2 + C\chi_n \, .
\]
Let $a$ be as in the statement of Lemma~\ref{lm:zn2_as_bound}. Since
$\lim n^{a} \norm{z_n^N}^2 = 0$ a.s., and furthermore, since
$\chi_n^{-1/2} \sim n^{\alpha - 1/2}$, we get that
$\chi_n^{-1/2} \norm{z_n^N}^2 \to_n 0$ a.s., and Proposition~\ref{zsympa} is
proven.
\end{proof}

\subsubsection{Theorem~\ref{th:avoid_trap} - End of Proof}\label{sec:avoid_trap_end}

We are now in the position to finish the proof of Theorem~\ref{th:avoid_trap}.
Indeed, recall from the discussion in the end of Section~\ref{sec:app_sgd} that the proof will be finished if we prove, for every $N \in \bbN$, that $\bbP(y_n^N \rightarrow 0) = 0$. To that end, we will verify that for every $N \in \bbN$, the sequence $(y_n^N)$ satisfies the assumptions of Proposition~\ref{prop:duflo}.

From Proposition~\ref{yduf}, it is obvious that the assumptions
\ref{Eeta0})--\ref{eq:unstable_exit}) in the statement of
Proposition~\ref{prop:duflo} are verified by $(\tilde{\eta}^N_{n})$.
Furthermore,
\begin{equation*}
  \sum_{i=N}^{\infty} \bbE[\norm{\varrho_{i}^N}^2] \leq C \sum_{i=N}^{\infty} \gamma_i^2 \bbE[(1+ \norm{\eta_{i}}^4)] \leq C \chi_N < +\infty \, ,
\end{equation*}
which shows that the assumption~\ref{sum-rho}) of Proposition~\ref{prop:duflo}
is satisfied by $(\varrho_{n}^N)$.
We also have from Proposition~\ref{yduf} that for all $n \geq N$,
\begin{equation*}
  \bbE_{n} \norm{\tilde{\varrho}^N_{n+1}} \leq C \norm{z_n^N} , \quad
  \bbE_{n} \norm{\tilde{\varrho}^N_{n+1}}^2 \leq C \norm{z_n^N}^2 , \quad
  \textrm{ and } \quad \bbE_{n} \norm{\tilde{\varrho}^N_{n+1}}^4 \leq C \norm{z_n^N}^4 \, .
\end{equation*}
Thus, Assumptions~\ref{rho-4mom})--\ref{rho-2mom}) of Proposition~\ref{prop:duflo} are verified by $(\tilde\varrho_{n}^N)$ by Lemma~\ref{lm:zn2_as_bound}. Finally, Assumption~\ref{hyp:dufl_sumtilde}) of Proposition~\ref{prop:duflo} is verified by Proposition~\ref{zsympa}.

As a consequence, we obtain for every $N \geq 0$ that
$\bbP(y_n^N \rightarrow 0) = 0$ and finally that $\bbP(x_n \rightarrow 0) = 0$.

%% file: proof_duflo.tex
\section{Proof of Proposition~\ref{prop:duflo}}\label{sec:proof_duf}

The proof of Proposition~\ref{prop:duflo} is technical and combines ideas from
\cite{bra-duf-96} and \cite{tarrespieges} (see also the older paper
\cite{pem-90}).

The main idea of \cite{bra-duf-96} and \cite{pem-90} was to observe that, due
to the fact that $J_D(0)$ has at least one eigenvalue with strictly negative
real part, the Ordinary Differential Equation (ODE) $\dot{\sy}(t) = -
D(\sy(t))$ admits a so-called center-stable invariant manifold of dimension
strictly less than $d$.  This manifold satisfies the following property: if at
some neighborhood of zero a solution of the ODE starts at a point outside of
the manifold, then it diverges from the origin.  The idea of
\cite{bra-duf-96} was then to construct a function that, in some sense,
measures the distance of a point to the invariant manifold. The assumption on
the perturbation sequence $(\tilde{\eta}_{n})$, combined with some
probabilistic estimates, then shows that the distance of the iterates $(y_n)$
to this manifold never vanishes, which implies that these iterates stay away
from zero.

From a technical point of view our proof combines techniques of
\cite{bra-duf-96} and \cite{tarrespieges}. The use of the approach of
\cite{tarrespieges} in the second part of the proof circumvents an error that
is found in the approach of \cite{bra-duf-96}.

\textbf{Outline of the proof.}

\begin{itemize}

\item \textit{Center-stable manifold.} In Section~\ref{sec:cen_man} we begin by
recalling the center-stable manifold theorem, which first version dates back to
Poincar\'e. The main result here is that, after a linear basis change, the
center-stable manifold could be represented through a function $G: \bbR^{d^+}
\rightarrow \bbR^{d^-}$, where $d^{+}$ (resp. $d^{-}$) denotes the dimension of
invariant subspaces of $J_{D}(0)$ that are associated to eigenvalues with
non-negative (respectively negative) real parts. This function $G$ is the first
step of the construction of a "distance" to the center-stable manifold.

Using the function $G$, we then construct from $(y_n)$ an $\bbR^{d^-}$-valued
sequence $(w_n^{-})$, that, due to the presence of a negative eigenvalue in
$J_{D}(0)$, ought to be pushed away from the origin. In
Lemma~\ref{lm:wn}, we show that it satisfies the recursion:
\[
w_{n+1}^{-} = w_n^{-} + \gamma_n H_n w_n^{-} + \gamma_n
(e_{n+1} + r_{n+1} + \tilde r_{n+1}  \, ,
\]
where the noise sequences $(e_{n+1})$ $(r_{n+1})$, and $(\tilde r_{n+1})$
satisfy assumptions analogous to their analogues
$(\eta_{n+1})$, $(\varrho_{n+1})$, and $(\tilde\varrho_{n+1})$,
and where the sequence of matrices $(H_n)$ converges to a
matrix that has only eigenvalues with positive real-parts. The presence of
these matrices, combined with the noise $e_{n}$, will cause the iterates
$w^-_n$ to be pushed away from zero, and so will be the case of the
$y_n$.

The content of this section closely follows the work of \cite{bra-duf-96}, the
only difference lying in the presence of the sequence $(\tilde\rho_{n+1})$ in
Proposition~\ref{prop:duflo}.

\item \textit{The repulsive sequence $(U_n)$ and its properties
(Sections~\ref{sec:Un} and \ref{sec:tarres}).} From $w_n^-$ we
construct a one-dimensional, positive random variable $(U_n)$
which is just a
well-chosen norm of $w_n^{-}$. The goal is, thus, to prove that $\bbP([U_n
\rightarrow 0]) = 0$. The construction of $U_n$ appears in \cite[Proof of
Proposition 4]{bra-duf-96}, however, from this point, our technique of proof
starts to follow the one of \cite{tarrespieges}. The properties of $(U_n)$
that will be required to show the non-convergence of this sequence are
provided by Lemmas~\ref{lm:Un_def} to~\ref{lm:Un_upp_bound}. In these lemmas,
we respectively show that
\begin{align*}
    U_{n+1} - U_n &\geq
     \gamma_n \scalarp{a_n}{e_{n+1} + r_{n+1} + \tilde r_{n+1}} \, ,
  \quad (a_n) \ \text{adapted and bounded}, \\
     \bbE_n U_{n+1}^2 - U_{n}^2 &\gtrsim
  C \gamma_n^2 - \gamma_n U_n \bbE_n (\| r_{n+1} \| + \| \tilde r_{n+1} \|), \\
     (U_{n+1} - U_n)^2 &\lesssim
    \gamma_n^2 (1 + \| e_{n+1} \|^2 + \| r_{n+1} \|^2 + \| \tilde r_{n+1} \|^2)  .
\end{align*}
These inequalities are then used to establish the following facts:
\begin{itemize}
 \item For a well-chosen $L >0$, for every $N \in \bbN$, the probability that
  for some $n \geq N$, $U_n \geq \sqrt{L \chi_n}$ is lower bounded
 (Lemma~\ref{lm:tarres_lemma1}).
    \item As soon as $U_n \geq \sqrt{L \chi_n}$, then with positive
   probability $\sup_{k \geq n} U_{k} \geq \sqrt{L \chi_n}/2$
 (Lemma~\ref{lm:tarres_lemma2}).
  \end{itemize}
Both of these lemmas are strengthened versions of \cite[Lemmas 1 and
2]{tarrespieges} (to our knowledge, first these ideas were presented in
\cite{pem-90}), where the terms $\tilde r_{n+1}$ are absent.
As in \cite{tarrespieges}, combining these two points, Lemma~\ref{lm:duf_end_proof} shows that
$\bbP([U_n \rightarrow 0]) = 0$ and completes the proof.
\end{itemize}

\subsection{Application of the Center-stable Manifold Theorem}
\label{sec:cen_man}

Consider the map $D: \bbR^d \rightarrow \bbR^d$ introduced before the statement
of Proposition~\ref{prop:duflo}, and consider the following ODE starting in a
neighborhood of zero
\[
  \dot{\sy}(t) = - D(\sy(t)) \, .
\]
Recalling the spectral factorization of the Jacobian $J_{D}(0)$ provided
before the statement of Proposition~\ref{prop:duflo} , it will be
convenient to work in the basis of the columns of $P$ by making the variable
change
\[
y \mapsto {\bfy} =
  \begin{bmatrix} y^+ \\ y^- \end{bmatrix} =
  P^{-1} y ,
\]
where $y^\pm \in \bbR^{d^\pm}$. With this at hand, writing
$\tD(\bfy) = P^{-1} D (P \bfy)$, the former ODE becomes in the new basis
\begin{equation}
\label{eq:cen_manifold_ode}
\dot{\underline{\bf y}}(t) =
  \begin{pmatrix}
    \dot{{\sy}}^{+}(t) \\
    \dot{{\sy}}^{-}(t)
  \end{pmatrix} = - \tD(\underline{\bf y}(t)) =
   -\begin{pmatrix}
      J^+ {\sy}^{+}(t) \\
      J^{-} {\sy}^{-}(t)
    \end{pmatrix} + R(\underline{\bf y}(t)) \, ,
\end{equation}
where ${\sy}^{\pm}(t) \in \bbR^{d^\pm}$, and
where $R: \bbR^d \rightarrow \bbR^d$ is $C^3$ on a neighborhood of zero with
$J_R(0) = 0$. The following classical proposition states that close to the
origin, the ODE~\eqref{eq:cen_manifold_ode} admits a center-stable, invariant
manifold.

\begin{proposition}[{\cite[Theorem 1]{kelley_stable66}}]
\label{prop:center_man_kelley}
  There is $\cU \subset \bbR^{d^+}$, a neighborhood of $0$, and a $C^2$
  function $G: \cU \rightarrow \bbR^{d^-}$ such that the following holds.
\begin{enumerate}
 \item It holds that $G(0) = 0$ and $J_{G}(0) = 0$.
 \item The set $\cV = \{(y^+, y^{-}) : y^{+} \in \cU, y^{-} = G(y^+) \}$ is
invariant for the ODE~\eqref{eq:cen_manifold_ode}.
In other words, if $\underline{\bfy}(t)$ is any solution to the
ODE~\eqref{eq:cen_manifold_ode} that starts at $\cV$, there is $t_0>0$ such
that for all $t \in (-t_0, t_0)$, $\underline{\bfy}(t) \in \cV$.
  \end{enumerate}
\end{proposition}

\begin{remark}\label{rmk:G_cenman_exten}
Taking, if necessary, a smaller neighborhood $\cU$ and extending $G$ outside of
it with the help of the Tietze extension theorem, we can always assume that $G$
is defined, is $C^2$ on the whole space $\bbR^{d^+}$ (but $\cV$ is still
defined only for $y^+ \in \cU$), and $\sup_{y^+ \in \bbR^{d^+}}\norm{J_{G}}\leq
c_G$, where $c_G > 0$ can be chosen as small as desired.
\end{remark}

\begin{lemma}\label{lm:cenm_JG_D}
Denoting $\tD = (\tD^+, \tD^-)$ and $\bfy = (y^{+}, y^{-})$, the second
property of Proposition~\ref{prop:center_man_kelley} implies that if
$\bfy \in \cV$, then $ \tD^{-}(\bfy) = J_{G}(y^+)\tD^{+}(\bfy)$.
\end{lemma}
\begin{proof}
Consider a point $\bfy = (y^{+}, y^{-}) \in \cV$ and let
$\underline{\bfy}(t) = ({\sy}^{+}(t), {\sy}^{-}(t))$ be the solution to
the ODE~\eqref{eq:cen_manifold_ode} that starts at $\bfy$. Noticing that,
for $t$ small enough, ${\sy}^{-}(t) = G({\sy}^{+}(t))$, and
differentiating this expression at zero completes the proof.
\end{proof}
Let $\cU$ and $G$ be the ones of Proposition~\ref{prop:center_man_kelley}, by
Remark~\ref{rmk:G_cenman_exten} we can assume that $G$ is $C^2$ on
$\bbR^{d^+}$. For $\bfy = (y^+, y^-) \in \bbR^d$ such that $y^{+} \in \cU$,
make a change of variable $(w^{+}, w^{-}) = ( y^+, y^{-} - G(y^+) )$.
Furthermore, define $F: \bbR^d \rightarrow \bbR^{d^-}$ as
\begin{align*}
      F(w) &= - \tD^{-}(\bfy) + J_{G}(y^+) \tD^{+}(\bfy) \\
      &= - \tD^{-}((w^+, w^- + G(w^+)) + J_{G}(w^+)
     \tD^{+}((w^+, w^- + G(w^+))\, .
\end{align*}

\begin{lemma}\label{lm:duflo_construction_F}
  There is a function $\Delta : \bbR^d \rightarrow \bbR^{d^{-} \times d^{-}}$ such that $\Delta(w) = \cO(\norm{w})$ and it holds that:
  \begin{equation*}
    F(w) = F(w^{+}, w^{-}) = (-J^- + \Delta(w)) w^- \, .
  \end{equation*}
\end{lemma}
\begin{proof}
  Let $\cU$ be the neighborhood of Proposition~\ref{prop:center_man_kelley}. By Lemma~\ref{lm:cenm_JG_D}, for any $w = (w^+, w^{-}) \in \bbR^d$ such that $w^{+} \in \cU$, it holds that $F(w^+, 0) = 0$. Therefore,
  \begin{equation}\label{eq:dufl_constr_1}
    F(w^+, w^-) = F(w^+, w^-) - F(w^+,0) = \int_{0}^{1} \partial_{-} F(w^+, t w^{-}) w^{-} \dif t \, ,
  \end{equation}
  where $\partial_{-} F(w^+, t w^{-})$ is the Jacobian of $w^{-} \mapsto F(w^{+}, w^-)$ (for fixed $w^+$) at $tw^{-}$. By a straightforward computation we obtain that:
  \begin{equation}\label{eq:dufl_constr_2}
    \partial_{-} F(0, 0) = -J^{-} \, .
  \end{equation}
  Furthermore, since $F$ is $C^2$, its Jacobian is Lipschitz on some neighborhood of zero, which implies that there is $C \geq 0$, such that for $w$ close enough to zero and for $t \in [0, 1]$,
  \begin{equation}\label{eq:dufl_constr_3}
    \norm{ \partial_{-} F(w^{+}, tw^{-}) - \partial_{-} F(0, 0)} \leq C  \norm{w} \, .
  \end{equation}
  Thus, the proof is completed by combining Equations~\eqref{eq:dufl_constr_1}--\eqref{eq:dufl_constr_3}.
\end{proof}

Before using these results, it will be convenient to strengthen our assumptions
on the field $D$ and noise sequences provided in the statement of
Proposition~\ref{prop:duflo}.  These simplifications are frequently used in the
field of stochastic approximation since the work of Lai and Wei~\cite{lai_wei}.
The following lemma is proven in the appendix.
\begin{lemma}
\label{duflo-fort}
The assumptions of Proposition~\ref{prop:duflo} can be replaced with the
following assumptions.  There exist constants $C_\times > 0$, and there
exist constants $c_\times > 0$ as small as needed, such that on the whole
probability space (and not only on $[y_n\to 0] \cap \Gamma$), the following holds true:
  \begin{enumerate}[i)]
 \item The map $\tD$ satisfies $\sup_{y\in \bbR^d} \|\tD(y)\| \leq C_D$.
 \item\label{Delta-small}
  The approximation of the function $F(w)$ provided by
  Lemma~\ref{lm:duflo_construction_F} satisfies
    $\sup_{w\in\bbR^d} \|\Delta(w)\| \leq c_F$.

    \item\label{hyp:dufl_constr_enzeri'} $\bbE_n \tilde\eta_{n+1} = 0$.
    \item\label{hyp:dufl_constr_e_uppi'}
   $\sup\bbE_n\norm{\tilde\eta_{n+1}}^4 \leq C_{\tilde\eta,1}$.
   \item\label{hyp:dufl_constr_e_lowi'}
   $\inf \bbE_n\norm{\tilde\eta_{n+1}^-} \geq C_{\tilde\eta,2}$.
   \item\label{sum-r'} $\sum_{i=0}^{\infty} \norm{\rho_{i+1}}^2 \leq c_\rho$.
   \item $\sup \bbE_n \norm{\tilde\rho_{n+1}}^4 \leq C_{\tilde\rho}$.
   \item $\sup \bbE_n \norm{\tilde\rho_{n+1}}^2 \leq c_{\tilde\rho,1}$.
   \item
     $\displaystyle{
  \sup \chi_n^{-1/2}
 \bbE_n \left[\sum_{i=n}^{\infty}
     \gamma_i \norm{\tilde\rho_{i+1}}\right] \leq c_{\tilde\rho,2}}$.
  \end{enumerate}
\end{lemma}

We are now in position to use the center-stable invariant manifold theorem to
construct our sequence $(w^-_n)$ that will be shown to stray away from zero.
Consider $G: \bbR^{d^{+}} \rightarrow \bbR^{d^{-}}$ as in
Proposition~\ref{prop:center_man_kelley}, with the extension provided by
Remark~\ref{rmk:G_cenman_exten}.  For all $n \in \bbN$, make
the basis change $\bfy_n = (y_n^+, y_n^-) = P^{-1} y_n$ where $(y_n)$ is the
sequence of iterates provided in the statement of Proposition~\ref{prop:duflo},
and write  $w_n = (w_n^{+}, w_n^{-}) = (y_n^{+}, y_n^{-} - G(y_n^{+}))$.  It is
obvious that $[y_n\to 0] \subset [w_n^- \to 0]$, which leads us to show in the
remainder of the proof that $\bbP(w^-_n \to 0) = 0$. The expression of
$w^-_n$ as an iterative system is provided by the following lemma.
\begin{lemma}
\label{lm:wn}
The adapted sequence $(w_n^-)$ is provided by the iteration
\[
    w_{n+1}^{-} = w_n^{-} +
   \gamma_n H_n w_n^{-} + \gamma_n (e_{n+1} + r_{n+1} + \tilde r_{n+1}) \, ,
\]
where the sequences $(e_n)$, $(r_n)$ and $(\tilde r_{n})$ are
$\bbR^{d^{-}}$--valued, $(H_n)$ is $\bbR^{d^{-} \times d^{-}}$--valued, and all
are adapted to $(\mcF_n)$.  Furthermore, there exists constants $C_\times > 0$
and constants $c_\times > 0$ that are as small as needed, such as the
following events hold with probability one:
  \begin{enumerate}[i)]
    \item\label{HJ} $\| H_n + J^- \| \leq c_H$.
    \item\label{Ene=0} $\bbE_n e_{n+1} = 0$.
    \item\label{Ene4} $\sup\bbE_n\norm{e_{n+1}}^4 \leq C_{e,1}$.
   \item\label{Ene>0}
   $\inf \bbE_n\norm{e_{n+1}} \geq C_{e,2}$.
   \item\label{sum-r-bnd} $\sum_{i=0}^{\infty} \norm{r_{i+1}}^2 \leq c_r$.
   \item $\sup \bbE_n \norm{\tilde r_{n+1}}^4 \leq C_{\tilde r}$.
   \item\label{Ert}
    $\sup \bbE_n \norm{\tilde r_{n+1}}^2 \leq c_{\tilde r,1}$.
   \item\label{En-sum-tr}
     $\displaystyle{
  \sup \chi_n^{-1/2}
 \bbE_n \left[\sum_{i=n}^{\infty}
      \gamma_i \norm{\tilde r_{i+1}}\right] \leq c_{\tilde r,2}}$.
  \end{enumerate}

\end{lemma}
\begin{proof}
Let $(\tilde\eta_n^+, \tilde\eta_n^-) = P^{-1} \tilde\eta_n$,
$(\varrho_n^+, \varrho_n^-) = P^{-1} \varrho_n$, and
$(\tilde\varrho_n^+, \tilde\varrho_n^-) = P^{-1} \tilde\varrho_n$ with the
obvious dimensions. It is clear that
\begin{equation*}
  y_{n+1}^{-} =  y_{n}^- - \gamma_n \tD^-(y_n)
  + \gamma_n \tilde{\eta}_{n+1}^{-} + \gamma_n \varrho^{-}_{n+1}
  + \gamma_n \tilde{\varrho}_{n+1}^{-} .
\end{equation*}
Furthermore,
\begin{equation*}
  \begin{split}
      G(y_{n+1}^+) &= G(y_{n}^+) + J_{G}(y_{n}^{+})(y_{n+1}^{+} - y_n^+)
        +   \xi(y_{n}^{+}, y_{n+1}^+) \\
       &=G(y_n^+) + \gamma_n  J_{G}(y_{n}^{+})( -\tD^{+}(y_n)
  + \tilde{\eta}_{n+1}^{+} + \varrho^{+}_{n+1} + \tilde{\varrho}_{n+1}^{+} )
    + \xi(y_{n}^{+}, y_{n+1}^+) \, ,
  \end{split}
\end{equation*}
where $\| \xi(y_{n}^{+}, y_{n+1}^+) \| \leq C \| y_{n+1}^{+} - y_n^+ \|^2$,
see Remark~\ref{rmk:G_cenman_exten}. We therefore have
\begin{align}
w^-_{n+1} &= y^-_{n+1} - G(y^+_{n+1}) \nonumber \\
 &= w^-_n - \gamma_n \left( \tD^-(y_n) - J_{G}(y_{n}^{+}) \tD^{+}(y_n) \right)
  + \gamma_n \left( e_{n+1} + r_{n+1} + \tilde r_{n+1} \right) \nonumber \\
 &= w^-_n + \gamma_n F(w_n)
  + \gamma_n \left( e_{n+1} + r_{n+1} + \tilde r_{n+1} \right) \nonumber \\
 &= w^-_n + \gamma_n H_n w^-_n
  + \gamma_n \left( e_{n+1} + r_{n+1} + \tilde r_{n+1} \right),
\label{w-n}
\end{align}
where $H_n = - J^- + \Delta(w_n)$, and $\Delta(w)$ is the function given by
Lemma~\ref{lm:duflo_construction_F} and controlled by
Lemma~\ref{duflo-fort}--\ref{Delta-small}). The bound~\ref{HJ}) above follows.
The other random variables at the right hand side of~\eqref{w-n} are given
as
\begin{equation*}
  \begin{split}
    e_{n+1} &= \tilde{\eta}_{n+1}^{-} - J_G(y_n^+)\tilde{\eta}_{n+1}^{+}
      \, , \\
    r_{n+1} &= \varrho_{n+1}^{-} - J_G(y_n^+)\varrho_{n+1}^{+} +
    \xi(y_n^+, y_{n+1}^{+})/\gamma_n \, , \\
    \tilde r_{n+1} &=
  \tilde{\varrho}_{n+1}^{-} - J_G(y_n^+)\tilde{\varrho}_{n+1}^{+}.
  \end{split}
\end{equation*}
Considering $e_{n+1}$, \ref{Ene=0}) and~\ref{Ene4}) follow immediately from
Lemma~\ref{duflo-fort}--\ref{hyp:dufl_constr_enzeri'}) and
Lemma~\ref{duflo-fort}--\ref{hyp:dufl_constr_e_uppi'})
and the boundedness of $\| J_G \|$. To obtain~\ref{Ene>0}), we write
\begin{equation*}
  \begin{split}
    \bbE_n\norm{e_{n+1}} &\geq
   \bbE_n\norm{\tilde{\eta}_{n+1}^{-}}
      - \bbE_n\norm{J_G(y_n^{+}) \tilde{\eta}_{n+1}^{+}} \\
    &\geq \bbE_n\norm{\tilde{\eta}_{n+1}^{-}}
      - \sup_{y^+} \| J_G(y^{+}) \| \  \bbE_n\norm{\tilde{\eta}_{n+1}^+}
  \end{split}
\end{equation*}
Since $\| J_{G}(y_n^+) \|$ can be taken as small as desired by
Remark~\ref{rmk:G_cenman_exten}, \ref{Ene>0}) follows from
Lemma~\ref{duflo-fort}--\ref{hyp:dufl_constr_e_uppi'})
and Lemma~\ref{duflo-fort}--\ref{hyp:dufl_constr_e_lowi'}).

The conclusions on $(\tilde r_{n+1})$ follow from the analogous properties of
$(\tilde{\varrho}_{n+1})$ provided by Lemma~\ref{duflo-fort} and the
boundedness of $\| J_{G} \|$.

It remains to establish the bound~\ref{sum-r-bnd}) on $(r_n)$. The
contributions of the terms $\rho^-_{n+1}$ and $J_G(y_n^+)\varrho_{n+1}^{+}$
to the sum in~\ref{sum-r-bnd}) can be controlled by
Lemma~\ref{duflo-fort}--\ref{sum-r'}). Regarding
 $\xi(y_n^+, y_{n+1}^{+})/\gamma_n$, we have
\begin{equation*}
  \| \xi(y_n^+, y_{n+1}^{+}) \|/\gamma_n \leq
   C\gamma_n (\| \tD^+(y_n)\|^2
   + \| \tilde{\eta}_{n+1}^{+}\|^2 + \| \varrho_{n+1}^{+} \|^2 +
     \|\tilde{\varrho}_{n+1}^{+}\|^2) \, ,
\end{equation*}
and the contribution of this term is easily controlled by using the
different bounds provided by Lemma~\ref{duflo-fort}.
\end{proof}

To show that $\bbP( w^-_n \to 0 ) = 0$, we show the non-convergence to zero of
a  sequence $(U_n)$, where $U_n$ is a well chosen norm of $w^-_n$. As
in~\cite{bra-duf-96}, this construction goes as follows. Since the eigenvalues
of $J^-$ have all negative real parts, we know from a theorem of Lyapunov
\cite[Cor.~2.2.4]{hor-joh-topics94} that there exists a positive definite
matrix $Q$ such that $Q J^- + (J^-)^\top Q = - 2 I_{d^-}$. With this at hand,
we set
\[
  U_n = \sqrt{\scalarp{w_n^{-}}{Q w_n^{-}}},
\]
in other words, $U_n$ is the norm of $w^-_n$ in the Euclidean space defined by
the scalar product $\scalarp{\cdot}{\cdot}_{Q}$, and is thus written $U_n = \|
w^-_n \|_Q$.  To deal with $(U_n)$, we build on the technique developed by
Tarr\`es in \cite{tarrespieges}.  Preliminary inequalities involving $U_n$ are
needed.

\subsection{Technical Inequalities on $U_n$}
\label{sec:Un}

In the following $\lambda_{\min}$ (respectively $\lambda_{\max}$) denotes the
minimal (respectively the maximal) eigenvalue of $Q$. In all the remainder,
we shall assume that the constant $c_H$ in Lemma~\ref{lm:wn}--\ref{HJ}) is
small enough so that $\| Q(H_n + J^-) \| \leq 1/2$.
During the proof, we shall keep track of some constants denoted as
$C_\times$, where $\times$ is the number of the lemma where these
constants are introduced.

\begin{lemma}\label{lm:Un_def}
   There is a $(\mcF_n)$-adapted, $\bbR^{d^{-}}$-valued sequence
 $(a_n)$ such that
\[
    U_{n+1} - U_n \geq
    \gamma_n \scalarp{a_n}{ e_{n+1} + r_{n+1} + \tilde r_{n+1}}
\]
 with probability one.
  Furthermore, $ \norm{a_n} \leq \lambda_{\max}/\lambda_{\min}$.
\end{lemma}

\begin{proof}
  As shown in \cite[page 405]{bra-duf-96}, for any two vectors
 $a,b \in \bbR^{d^-}$, with $a \neq 0$, it holds that:
  \begin{equation}\label{eq:scalarQ}
    \| a+b \|_Q - \| a \|_Q \geq
     \frac{\scalarp{a}{b}_Q}{\| a \|_Q} \, .
  \end{equation}
Furthermore, it is obvious that $\| b \|_Q \geq \scalarp{u}{b}_Q$, with $u$
being an arbitrary vector such that $\| u \|_Q = 1$. Therefore, if $U_n = 0$,
then we have from the previous lemma that
  \begin{equation}\label{eq:Un_w_zer}
    \begin{split}
      U_{n+1} = U_{n+1} - U_n \geq \gamma_n
    \scalarp{u}{e_{n+1} + r_{n+1} + \tilde r_{n+1}}_Q \, .
    \end{split}
  \end{equation}
Otherwise, if $U_n \neq 0$, then, using Equation~\eqref{eq:scalarQ}, we obtain:
\begin{align*}
  U_{n+1} - U_n &\geq U_n^{-1} \scalarp{w_n^{-}}{w_{n+1}^{-} - w_n^-}_Q \\
 &=
 \gamma_n U_n^{-1} \scalarp{w_n^{-}}{Q H_n w^-_n} +
 \gamma_n U_n^{-1} \scalarp{w_n^{-}}{e_{n+1} + r_{n+1} + \tilde r_{n+1}}_Q .
\end{align*}
Noticing from the definition of the matrix $Q$ through Lyapounov's theorem
that $\scalarp{x}{- Q J^{-} x} = \| x \|^2$, and recalling that
$\| Q(H_n + J^-) \| \leq 1/2$, we obtain that $\scalarp{w_n^{-}}{Q H_n w^-_n}
\geq 0$, thus,
\[
  U_{n+1} - U_n \geq
 \gamma_n U_n^{-1} \scalarp{w_n^{-}}{e_{n+1} + r_{n+1} + \tilde r_{n+1}}_Q .
\]
The result follows by taking $a_n =  U_n^{-1} Q w_n^{-}$ if $U_n \neq 0$ and
$a_n = Q u$ otherwise.
\end{proof}

\begin{lemma}
 \label{lm:Un_low_bound}
It holds that
  \begin{equation}
 \label{EnU-U}
    \bbE_n U_{n+1}^2 - U_{n}^2 \geq C_{\ref{lm:Un_low_bound}} \gamma_n^2 -
     2 \lambda_{\max} \gamma_n U_n ( \bbE_n \| r_{n+1} \| +
  \bbE_n \| \tilde r_{n+1} \| ),
  \end{equation}
  where $C_{\ref{lm:Un_low_bound}}  = \lambda_{\min} C_{e,2}^2 / 2$
from Lemma~\ref{lm:wn}--\ref{Ene>0}).
\end{lemma}
\begin{proof}
We have
\begin{align*}
  U_{n+1}^2 - U_n^2 &= \|w_{n+1}^{-}\|^2_Q - \|w_n^-\|^2_Q
   = 2 \scalarp{w_{n}^{-}}{w_{n+1}^{-} - w_n^{-}}_Q
   + \| w_{n+1}^{-} - w_n^-\|_Q^2 \\
&= 2\gamma_n \scalarp{w_n^{-}}{Q H_{n} w_n^{-}}
  + 2\gamma_n \scalarp{w_n^{-}}{Q e_{n+1}}
   + 2\gamma_n \scalarp{w_n^{-}}{r_{n+1} + \tilde r_{n+1}}_Q
   + \| w_{n+1}^{-} - w_n^-\|_Q^2 \\
&\geq
   2\gamma_n \scalarp{w_n^{-}}{Q e_{n+1}}
   + 2\gamma_n \scalarp{w_n^{-}}{r_{n+1} + \tilde r_{n+1}}_Q
   + \| w_{n+1}^{-} - w_n^-\|_Q^2 ,
\end{align*}
remembering that $\scalarp{w_n^{-}}{Q H_{n} w_n^{-}} \geq 0$ as in the previous
proof. Moreover,
\begin{align*}
\| w_{n+1}^{-} - w_n^-\|_Q^2 &= \gamma_n^2 \| H_n w^-_n + e_{n+1} +
  r_{n+1} + \tilde r_{n+1} \|_Q^2 \\
&= \gamma_n^2 \| e_{n+1} \|_Q^2 + \gamma_n^2 \| H_n w^-_n +
  r_{n+1} + \tilde r_{n+1} \|_Q^2 + 2\gamma_n^2
 \scalarp{e_{n+1}}{H_n w^-_n + r_{n+1} + \tilde r_{n+1}}_Q \\
 &\geq \gamma_n^2 \| e_{n+1} \|_Q^2
+ 2\gamma_n^2 \scalarp{e_{n+1}}{H_n w^-_n}_Q
+ 2\gamma_n^2 \scalarp{e_{n+1}}{r_{n+1} + \tilde r_{n+1}}_Q .
\end{align*}
Since $\bbE_n e_{n+1} = 0$, we obtain that
\[
  \bbE_n U_{n+1}^2 - U_n^2 \geq
  \gamma_n^2 \bbE_n \| e_{n+1} \|_Q^2
- 2\gamma_n^2
   \bbE_n \left| \scalarp{e_{n+1}}{r_{n+1} + \tilde r_{n+1}}_Q \right|
   - 2\gamma_n
  \bbE_n \left| \scalarp{w_n^{-}}{r_{n+1} + \tilde r_{n+1}}_Q \right|  .
\]
Using Lemma~\ref{lm:wn}, we have $\bbE_n  \| e_{n+1} \|_Q^2 \geq
 \lambda_{\min} \bbE_n \| e_{n+1} \|^2 \geq \lambda_{\min} C_{e,2}^2$.
Moreover,
\[
\bbE_n \left| \scalarp{e_{n+1}}{r_{n+1} + \tilde r_{n+1}}_Q \right|
\leq (\bbE_n \| e_{n+1} \|_Q^2)^{1/2}
  (\bbE_n \| r_{n+1} + \tilde r_{n+1} \|_Q^2)^{1/2},
\]
which can be made as small as wished thanks to the bounds given by \ref{Ene4}),
\ref{sum-r-bnd}) and \ref{Ert}) in the statement of Lemma~\ref{lm:wn}, and
played against $\bbE_n \| e_{n+1} \|^2$ to provide the term
$C_{\ref{lm:Un_low_bound}} \gamma_n^2$ at the right hand side of
Inequality~\eqref{EnU-U}. We also have
 $\left| \scalarp{w_n^{-}}{r_{n+1} + \tilde r_{n+1}}_Q \right|
\leq \| w_n^{-} \|_Q ( \| r_{n+1} \|_Q + \| \tilde r_{n+1} \|_Q )
\leq \lambda_{\max} U_n  ( \| r_{n+1} \| + \| \tilde r_{n+1} \| )$, which
proves the lemma.
\end{proof}

Recall that $\chi_n := \sum_{i=n}^{\infty} \gamma_i^2$.
\begin{lemma}\label{lm:Un_upp_bound}
  There is a constant $C_{\ref{lm:Un_upp_bound}} >0$ such that if for
  $n \in \bbN$, $U_n$ is such that $U_n^2 \leq L \chi_n$ for some $L > 0$, then
\[
    (U_{n+1} - U_n)^2 \leq  C_{\ref{lm:Un_upp_bound}} \gamma_n^2
    (L +\norm{r_{n+1}}^2 + \norm{r'_{n+1}}^2 + \norm{e_{n+1}}^2) \, .
\]
\end{lemma}
\begin{proof}
We have
  \begin{equation*}
    \begin{split}
       (U_{n+1} - U_n)^2 &= (\| w^-_{n+1} \|_Q - \| w^-_{n} \|_Q )^2
   \leq \| w_{n+1}^{-}  - w_n^{-} \|_Q^2 \\
          &\leq C \gamma_n^2 (\norm{H_n w_n^{-}}_Q^2 + \norm{e_{n+1}}_Q^2 +
    \norm{r_{n+1}}_Q^2 + \norm{\tilde r_{n+1}}_Q^2 ) \\
           &\leq C \gamma_n^2 (U_n^2 + \norm{e_{n+1}}^2 +
   \norm{r_{n+1}}^2 +\norm{\tilde r_{n+1}} ) \\
           &\leq C \gamma_n^2 (L \chi_{0} + \norm{e_{n+1}}^2 +
   \norm{r_{n+1}}^2 +\norm{\tilde r_{n+1}} ),
    \end{split}
  \end{equation*}
  where the last inequality comes from the fact that
   $U_n^2 \leq L \chi_n \leq L \chi_0$.
\end{proof}

\subsection{Proof of Proposition~\ref{prop:duflo} by Proving that
$\bbP(U_n \to 0) = 0$}
\label{sec:tarres}

For $N \in \bbN$ and $L >0$, denote
\begin{equation*}
  \tau_N(L) := \inf\{ k \geq N: U_k^2 \geq L \chi_k \} \, .
\end{equation*}

The following lemma is an adaptation of \cite[Lemma 1]{tarrespieges} to our setting.

\begin{lemma}\label{lm:tarres_lemma1}

For $L >0$, if the quantity $\max(c_{\tilde r,2},c_r^{1/2})\sqrt{L}$ is small
enough, then there is a constant $C_{\ref{lm:tarres_lemma1}} =
 C_{\ref{lm:tarres_lemma1}}(L) > 0$ such that for all $N \in \bbN$,
$\bbP_N(\tau_{N}(L) < +\infty) \geq C_{\ref{lm:tarres_lemma1}}(L)$.
\end{lemma}
\begin{proof}
Fixing $N, L$, with small notational abuse we will write in this proof $\tau :=
\tau_N(L)$. Notice that for $ N \leq n < \tau$, we have $U_n \leq \sqrt{L
\chi_n} \leq \sqrt{L \chi_N}$. Lemma~\ref{lm:Un_low_bound} shows that
the random process $(Z_n)_{n\geq N}$ defined as
\[
Z_n = U_n^2 - C_{\ref{lm:Un_low_bound}} \sum_{i=N}^{n-1} \gamma_{i}^2
  + 2 \lambda_{\max} \sum_{i=N}^{n-1} \gamma_{i} U_{i} ( \| r_{i+1} \| +
 \| \tilde r_{i+1} \| )
\]
is a submartingale. Thus, the stopped process $(Z_{n\wedge \tau})$ is
a $\mcF_n$--submartingale, and it holds that $\bbE_N[Z_{n \wedge \tau}] \geq Z_{N}$, which implies:
  \begin{equation*}
    \begin{split}
      \bbE_N[U^2_{n \wedge \tau}] \geq U_N^2 +
  C_{\ref{lm:Un_low_bound}} \sum_{i=N}^{n-1} \gamma_{i}^2
   \bbP_N(\tau > n) - 2\lambda_{\max} \sqrt{L \chi_N}
   \bbE_N\left[\sum_{i=N}^{+\infty} \gamma_i (\|r_{i+1}\| +
   \| \tilde r_{i+1}\|) \right] \, .
    \end{split}
  \end{equation*}

By Cauchy-Schwarz's inequality and the bound of
Lemma~\ref{lm:wn}--\ref{sum-r-bnd}), it holds that $\sum_{i=N}^{\infty}
\gamma_i \|r_{i+1}\| \leq \sqrt{\chi_N c_r}$.  Similarly, using
Lemma~\ref{lm:wn}--\ref{En-sum-tr}),  we obtain that
$\bbE_N[\sum_{i=N}^{\infty} \gamma_i \| \tilde r_{i+1}\| ]
\leq c_{\tilde r,2}\sqrt{\chi_N}$.  Therefore,
   \begin{equation*}
     \begin{split}
    \bbE_N[U^2_{n \wedge \tau} - U_N^2] \geq
  &\left(C_{\ref{lm:Un_low_bound}} \sum_{i=N}^{n-1} \gamma_i^2 -
 2 \lambda_{\max} \sqrt{L c_r} \chi_N - 2\lambda_{\max} c_{\tilde r,2}
     \sqrt{L} \chi_N\right) \bbP_N(\tau > n) \\
         &- \left(  2\lambda_{\max} \sqrt{L c_r} \chi_N
    + 2\lambda_{\max} c_{\tilde r,2} \sqrt{L} \chi_N\right)
      \bbP_N(\tau \leq  n)\, .
     \end{split}
   \end{equation*}
Thus, if $\lambda_{\max} \max(c_{\tilde r,2},c_r^{1/2})\sqrt{L} \leq
  C_{\ref{lm:Un_low_bound}}/8$, which can be assumed due to the expression
 of $C_{\ref{lm:Un_low_bound}}$ provided by Lemma~\ref{lm:Un_low_bound}, then,
   \begin{equation}\label{eq:tau_low}
     \begin{split}
   \bbE_N[U^2_{n \wedge \tau} - U_N^2] &\geq  C_{\ref{lm:Un_low_bound}}
  \left( (\chi_{N} - \chi_{n}) - \frac{1}{2} \chi_N \right) \bbP_N(\tau > n)
  - \frac{C_{\ref{lm:Un_low_bound}} \chi_N}{2} \bbP_N(\tau \leq n) \\
             & \geq C_{\ref{lm:Un_low_bound}}
  \left( \frac{\chi_N - 2\chi_{n}}{2}\right) \bbP_N(\tau > n)
     - \frac{C_{\ref{lm:Un_low_bound}} \chi_N}{2} \bbP_N(\tau \leq n) \, .
     \end{split}
   \end{equation}
   On the other hand, if $\tau = N$, then $U_{n \wedge \tau}^2 - U_N^2 = 0$ and if $\tau > N$, then
   \begin{equation}\label{eq:Un_tau_interm}
     \begin{split}
            U_{n \wedge \tau}^2 - U_N^2 &\leq U_{n \wedge \tau}^2 \\
            &\leq U_{n}^2 \1_{\tau > n} +  U_{\tau}^2 \1_{\tau \leq n}  \\
            &\leq L \chi_n \1_{\tau >n } + 2(U_{\tau - 1}^2 + (U_{\tau} - U_{\tau -1})^2) \1_{\tau \leq n} \\
            &\leq L \chi_n \1_{\tau >n } + 2 L\chi_N \1_{\tau \leq n} + (U_{\tau} - U_{\tau -1})^2\1_{\tau \leq n} \, ,
     \end{split}
   \end{equation}
   where the last inequality follows from the fact that
 $U_{\tau - 1} \leq L\chi_{\tau - 1} \leq L \chi_N$.

 Using Lemma~\ref{lm:Un_upp_bound}, we also obtain:
   \begin{equation}\label{eq:Utau_diff}
     \begin{split}
     \bbE_{N}[(U_{\tau} - U_{\tau -1})^2\1_{N < \tau \leq n}] &\leq
  C_{\ref{lm:Un_upp_bound}} \bbE_{N} \left[\sum_{ i= N+1}^{n} \1_{\tau = i}
  \gamma_i^2 \left(L + \norm{e_i}^2 + \norm{r_{i}}^2 + \norm{\tilde r_{i}}^2
    \right)\right] \\
     &\leq  C_{\ref{lm:Un_upp_bound}}  (L+ c_r) \chi_{N} \bbP(\tau \leq n)
   + C_{\ref{lm:Un_upp_bound}} \bbE_{N}\left[\sum_{i=N+1}^{n}\gamma_i^2
   \1_{\tau = i} (\norm{e_{i}}^2  + \norm{\tilde r_{i}}^2)\right]
        \end{split}
   \end{equation}
 By the Cauchy-Schwarz inequality, we obtain that
  $\bbE_{N}[\1_{\tau = i} \norm{\tilde r_{i}}^2] \leq
\bbP_{N}(\tau = i)^{1/2} C_{\tilde r}^{1/2}$ from Lemma~\ref{lm:wn}, and
similarly, $\bbE_{N}[\1_{\tau = i} \norm{e_{i}}^2] \leq C_{e,1}^{1/2}
  \bbP_{N}(\tau = i)^{1/2}$.  Thus, combining
Equations~\eqref{eq:Un_tau_interm} and \eqref{eq:Utau_diff}, we obtain:
   \begin{equation}\label{eq:tau_upp}
     \begin{split}
   \bbE_{N}[U_{n \wedge \tau}^2 - U_N^2 ] &\leq
  L \chi_n \bbP_{N}(\tau > n) + (2L + C_{\ref{lm:Un_upp_bound}} L
  + C_{\ref{lm:Un_upp_bound}} c_r)\chi_{N} \bbP_{N}(\tau \leq n) \\
   &\phantom{=}
   + C_{\ref{lm:Un_upp_bound}} (C_{e,1}^{1/2} + C_{\tilde r}^{1/2})
    \chi_{N} \bbP_{N}(\tau \leq n)^{1/2}\, .
     \end{split}
   \end{equation}

  Finally, combining Equations~\eqref{eq:tau_low} and \eqref{eq:tau_upp}, we
 obtain:
   \begin{equation*}
     \begin{split}
   C_{\ref{lm:Un_low_bound}} \left( \frac{\chi_N - 2\chi_{n}}{2}\right)
  \bbP_N(\tau > n) &\leq L \chi_n \bbP_{N}(\tau > n)
  + (2L + C_{\ref{lm:Un_upp_bound}} L + C_{\ref{lm:Un_upp_bound}} c_r
   + C_{\ref{lm:Un_low_bound}}/2)\chi_{N} \bbP_{N}(\tau \leq n) \\
        &+ C_{\ref{lm:Un_upp_bound}}
  (C_{e,1}^{1/2} + C_{\tilde r}^{1/2}) \chi_{N} \bbP_{N}(\tau \leq n)^{1/2}\, .
     \end{split}
   \end{equation*}
And letting $n$ tend to infinity, we obtain:
\begin{equation*}
  \frac{C_{\ref{lm:Un_low_bound}}}{2} \bbP_{N}(\tau=  \infty)\leq \left( 2L + C_{\ref{lm:Un_upp_bound}} L + C_{\ref{lm:Un_upp_bound}} c_r + C_{\ref{lm:Un_low_bound}}/2 + C_{\ref{lm:Un_upp_bound}} (C_{e,1}^{1/2} + C_{\tilde r}^{1/2})\right) (1-\bbP_{N}(\tau =  \infty))^{1/2} \, .
\end{equation*}
Similarly to \cite[Proof of Lemma 1]{tarrespieges}, this inequality shows the existence of a constant $C_{\ref{lm:tarres_lemma1}}$, that depends only on $L, C_{\ref{lm:Un_low_bound}}, C_{\ref{lm:Un_upp_bound}}, C_{e,1}, c_r, C_{\tilde r}$ such that $\bbP_{N}(\tau(N, L) = \infty) \leq(1- C_{\ref{lm:tarres_lemma1}}(L)) $, which completes the proof.
\end{proof}

The following lemma is an adaptation of \cite[Lemma 2]{tarrespieges} to our setting.

\begin{lemma}\label{lm:tarres_lemma2}
  If $L$ is chosen such that
  \begin{equation*}
    \sqrt{L}\geq  \frac{\lambda_{\max}}{\lambda_{\min}}\left( 2c_r^{1/2} + 16 C_{e,1}^{1/4}+ 16 c_{\tilde r,2} \right) \, ,
  \end{equation*}
  then for all $N \in \bbN$ and $n \geq N$,
\begin{equation*}
  \bbP_n(\liminf U_n >0)\1_{\tau_N(L) = n} \geq \frac{1}{2}\1_{\tau_N(L) = n} \, .
\end{equation*}

\end{lemma}
\begin{proof}
The idea of the proof is to use Lemma~\ref{lm:Un_def} and to show that
there is some $0 <b < \sqrt{L}/2$ such that, with probability at least $1/2$,
the norm of $\sum_{i=n}^{\infty}\gamma_n \scalarp{a_i}{e_{i+1} + r_{i+1}
+ \tilde r_{i+1}}$ will be less that $b \sqrt{\chi_n}$. Since, when
$\tau_{N}(L) = n$, it holds that $U_{n} \geq \sqrt{L}$, this will show that
$\liminf U_n >0$ at least with probability 1/2.  To simplify the notations we
will denote in this proof
$$
E_{n+1}, {R}_{n+1}, \tilde{R}_{n+1} =
\scalarp{a_n}{e_{n+1}}, \scalarp{a_n}{r_{n+1}}, \scalarp{a_n}{\tilde r_{n+1}}
  \, .
$$

Fix $N \in \bbN$, $n \geq N$ and notice that by Doob's inequality:
\begin{equation*}
  \bbE_{n}\left[ \sup_{k \geq n}\left|\sum_{i=n}^{k} \gamma_i E_{i+1}\right|^2 \right] \leq 4 \bbE_{n}\left[ \sum_{i=n}^{\infty} \gamma_i^2 E_{i+1}^2\right] \leq 4 \frac{\lambda_{\max}^2}{\lambda_{\min}^2}C_{e,1}^{1/2} \chi_n \, ,
\end{equation*}
where we have used the fact that $\bbE_n[\norm{e_{n+1}}^4] \leq C_{e,1}$ and $\lambda_{\min}\norm{a_n} \leq \lambda_{\max}$.
Moreover,
\begin{equation*}
  \bbE_n \left[ \sup_{k \geq n}
  \left|\sum_{i=n}^{k} \gamma_i \tilde{R}_{i+1}\right|\right]\leq
\norm{a_n}\bbE_n\left[\sum_{i=n}^{\infty} \gamma_i \norm{\tilde r_{i+1}}
  \right]
  \leq \frac{ c_{\tilde r,2} \lambda_{\max}}{\lambda_{\min}} \sqrt{\chi_n} \, .
\end{equation*}
Therefore, for any constant $b >0$, applying Markov's inequality, we obtain:
\begin{equation*}
  \bbP_{n}\left( \inf_{k \geq n}\sum_{i=n}^{k} \gamma_i E_{i+1}
  < - b \sqrt{\chi_n} \right)
 \leq  \bbP_{n} \left( \sup_{k \geq n}
 \left|\sum_{i=n}^{k} \gamma_i E_{i+1}\right|^2 \geq b^2 \chi_{n}\right)
 \leq \frac{ 4 \lambda_{\max}^2C_{e,1}^{1/2}}{\lambda_{\min}^2b^2} \, ,
\end{equation*}
and
\begin{equation*}
  \bbP_{n} \left( \inf_{k \geq n} \sum_{i=n}^{k} \gamma_i \tilde{R}_{i+1} < - b \sqrt{\chi_n} \right) \leq \bbP_{n}\left(\left|\sup_{k \geq n} \sum_{i=n}^{k}\gamma_i \tilde{R}_{i+1}\right| \geq b \sqrt{\chi_n} \right) \leq \frac{c_{\tilde r,2}\lambda_{\max} }{\lambda_{\min}b} \, .
\end{equation*}

Thus, denoting $\Gamma_1$ the event:
\begin{equation*}
  \Gamma_1 := \left[ \inf_{k \geq n} \sum_{i=n}^{k} \gamma_i \tilde{R}_{i+1} < - b \sqrt{\chi_n}\right] \bigcup \left[ \inf_{k \geq n}\sum_{i=n}^{k} \gamma_i E_{i+1} < - b \sqrt{\chi_n}\right]\,
\end{equation*}
and fixing $b = 4  \lambda_{\max}/\lambda_{\min} \left(C_{e,1}^{1/4} + c_{\tilde r,2} \right)$, we obtain
\begin{equation*}
  \bbP_{n}(\Gamma_1 ) \leq \frac{ 4 \lambda_{\max}^2C_{e,1}^{1/2}}{\lambda_{\min}^2b^2} + \frac{ c_{\tilde r,2}\lambda_{\max} }{\lambda_{\min}b} \leq \frac{1}{2} \, .
\end{equation*}
Moreover, for such $b$, using Lemma~\ref{lm:Un_def}, we obtain for $k \geq n$,
\begin{equation*}
  (U_{k} - U_n)\1_{\Gamma_1^c} \geq \1_{\Gamma_1^c}\sum_{i=n}^k \gamma_i (E_{i+1} + R_{i+1} + \tilde{R}_{i+1}) \geq
 \1_{\Gamma_1^c}\sum_{i=n}^{k} \gamma_i R_{i+1}
  - 2 b \sqrt{\chi_n} \1_{\Gamma_1^c} \, .
\end{equation*}
Furthermore, by the Cauchy-Schwarz inequality, it also holds that:
\begin{equation*}
  \sup_{k \geq n}\left|\sum_{i=n}^k \gamma_i R_{i+1}\right| \leq \frac{\lambda_{\max}}{\lambda_{\min}}\sum_{i=n}^{\infty} \gamma_i \norm{r_{i+1}} \leq \frac{\lambda_{\max}}{\lambda_{\min}}\chi_{n}^{1/2} c_r^{1/2} \, ,
\end{equation*}
and finally, for $k \geq n$, we obtain,
\begin{equation*}
  U_{k} \1_{\Gamma_1^c} \geq U_n \1_{\Gamma_1^c} -(c_r^{1/2}\lambda_{\max}/\lambda_{\min}  + 2 b)  \sqrt{\chi_n} \1_{\Gamma_1^c} \, .
\end{equation*}
In Particular, if $\sqrt{L} \geq 2 c_r^{1/2} \lambda_{\max}/\lambda_{\min}+ 4b$, then on the event $\Gamma_1^c \cap [\tau_{N}(L) = n]$ it holds
\begin{equation*}
  \begin{split}
      U_{k}&\geq  U_n - \sqrt{L}/2  \sqrt{\chi_n} \geq \left( \sqrt{L} - \sqrt{L}/{2}\right) \sqrt{\chi_n} \geq \sqrt{L}\sqrt{\chi_n}/2 \, .
  \end{split}
\end{equation*}
This shows that $\Gamma_1^c \cap [\tau_{N}(L) = n] \subset [\liminf U_n >0] \cap [\tau_{N}(L) = n]$. Thus,
\begin{equation*}
  \begin{split}
     \bbP_n([\liminf U_n >0] )\1_{\tau_{N}(L) = n} &\geq \bbP_n(\Gamma_1^c \cap [\tau_{N}(L) = n] ) \\
     &= \bbP_n(\Gamma_1^c) \1_{\tau_{N}(L) = n} \\
     &\geq \frac{1}{2} \1_{\tau_{N}(L) = n} \, ,
  \end{split}
\end{equation*}
which completes the proof.
\end{proof}

We are now in position to complete the proof of Proposition~\ref{prop:duflo}
by proving that $\bbP(U_n\to 0) = 0$. This will be the content of the
following lemma.

\begin{lemma}\label{lm:duf_end_proof}
\label{Un->0} $\bbP(U_n\to 0) = 0$.
\end{lemma}
\begin{proof}
We proceed as in \cite{tarrespieges}. Indeed, we can always choose $c_r,
c_{\tilde r,2}$ and $L$ respectively small and large enough such that the
prerequisites of Lemmas~\ref{lm:tarres_lemma1} and \ref{lm:tarres_lemma2} are
satisfied. Thus, by Lemma~\ref{lm:tarres_lemma1}, there is
$C_{\ref{lm:tarres_lemma1}}(L)$ such that $\bbP(\tau_{N}(L) < +\infty) \geq
C_{\ref{lm:tarres_lemma1}}(L) >0 $. Then, applying
Lemma~\ref{lm:tarres_lemma2}, we obtain:
\begin{equation*}
  \bbE_{N}[\1_{\liminf U_n >0}] \geq \sum_{i=N}^{\infty} \bbE_{N}[ \1_{\liminf U_n > 0} \1_{\tau_{N}(L) = i}] \geq \frac{1}{2} \sum_{i=N}^{\infty} \bbE_{N}[\1_{\tau_{N}(L)} = i] = \frac{C_{\ref{lm:tarres_lemma1}}(L)}{2} \, .
\end{equation*}
Since $[\liminf U_n >0] \in \mcF_{\infty} = \sigma(\bigcup_{i=0}^{\infty}
\mcF_i)$, we know by L\'evy's zero-one law (see e.g. \cite[Theorem 14.2]{Wil91})
that $\lim_{N \rightarrow \infty} \bbE_{N}[\1_{\liminf U_n >0}] =\1_{\liminf
U_n >0}$ almost surely. Thus, almost surely, $\1_{\liminf U_n >0} = 1$, which
shows that $\bbP([U_n \rightarrow 0]) = 0$.
\end{proof}

\section*{Acknowledgments}

We would like to thank the anonymous reviewers for their outstanding job of refereeing, and in particular of pointing out two serious flaws in the first version of this work. The work of Sholom Schechtman was supported by the “R\'egion Ile-de-France”.

%% file: app.tex
\begin{appendices}

\section{Definability and Whitney Stratifications}

\subsection{o-minimality}
\label{sec:o-min}

An o-minimal structure can be viewed as an axiomatization of diverse
properties of semialgebraic sets.
In an o-minimal
structure, pathological sets such as Peano curves or the graph of the
function $\sin \frac{1}{x}$ do not exist.
To our knowledge the first
work to link ideas between optimization and o-minimal structures was
\cite{bolte2007clarke}, where the authors analyzed the structure of the
Clarke subdifferential of a definable function and extended the
Kurdyka-\L{}ojasiewicz inequality \cite{kur_ongrad} to the nonsmooth
setting. Nowadays a rich body of literature enforces this link, see
e.g. \cite{dav-dru-kak-lee-19, lewis_small, bol_dan_lew09,
 attouch_semi, bolte2019conservative}. A nice exposure about
usefulness of o-minimal theory in optimization is \cite{iof08}.
Results on the Verdier and Whitney stratification of definable sets
can be found in \cite{cos02, van96, loi98}.

 An \emph{o-minimal structure} is a family $\cO = (\cO_n)_{n \in \bbN*}$, where $\cO_n$ is a set of subsets of $\bbR^n$, verifying the following axioms.

 \begin{enumerate}
   \item If $Q : \bbR^n \rightarrow \bbR$ is a polynomial, then $\{Q(x) = 0\} \in \cO_n$.
   \item If $A$ and $B$ are in $\cO_n$, then the same is true for $A \cap B$, $A \cup B$ and $\bbR^n \backslash A$.
   \item If $A \in \cO_n$ and $B \in \cO_m$, then $A \times B \in \cO_{n+m}$.
   \item If $A \in \cO_n$, then the projection of $A$ on its first ($n-1$) coordinates is in $\cO_{n-1}$.
   \item Every element of $\cO_1$ is exactly a finite union of intervals and points.
 \end{enumerate}
 Sets contained in $\cO$ are called \emph{definable}. We call a map
 $f : \bbR^k \rightarrow \bbR^m$ definable if its graph is definable.
 Definable sets and maps have remarkable stability
 properties, for instance, if $f$ and $A$ are definable, then $f(A)$
 and $f^{-1}(A)$, any composition of two functions definable in the
 same o-minimal structure is definable, and many others. These properties show that most of the functions that are used in optimization are definable. Examples of such are: semialgebraic functions, analytic functions restricted to a semialgebraic compact, exponential and logarithm (see e.g. \cite{bol_dan_lew09, bier_semi_sub, wil_o_min}). In particular, it can be shown that the loss of a neural network is a
 definable function \cite{dav-dru-kak-lee-19}.

 \subsection{Whitney Stratification}
 \label{sec:whitney}

 From the definition of $\bsd_a$ in Equation~(\ref{eq:angle_vect}), one can define the following distance between two vector spaces $E_1, E_2$:
 \begin{equation}
   \bsd(E_1,E_2) = \max \{ \bsd_a(E_1, E_2), \bsd_a(E_2,E_1)\} \, .
 \end{equation}
 \begin{definition}
 We say that a $C^p$ stratification $(S_i)$ satisfies a \emph{Whitney-(a) property}, if for every couple of distinct strata $S_i, S_j$, for each $y \in S_i \cap \overline{S_j}$ and for each sequence $(x_n)$ in $S_j$ such that $x_n \rightarrow y$, it holds:
 \begin{equation}
 \textrm{\emph{w-(a)}} \quad{} \textrm{ There is $E \subset \bbR^d$ such that } \quad \bsd(\cT_{x_n}S_j,E) \rightarrow 0 \implies \cT_{y}S_i \subset E \, .
\end{equation}
 We will refer to $(S_i)$ as a Whitney $C^p$ stratification.
\end{definition}
 It is known (see \cite{cos02,van96}) that every definable function
 $f$ admits a Whitney $C^p$ (for any $p$) stratification $(X_i)$ of
 its domain such that $f$ is $C^p$ on each stratum. The following
 ``projection formula'' relates the Clarke subdifferential
 $\partial f(y)$ of $f$ at $y$, to $\nabla_{X_i} f(y)$.

 \begin{lemma}[Projection formula, {\cite[Lemma 8]{bolte2007clarke}}]\label{lm:f_bol_strat}
   Let $f: \bbR^d \rightarrow \bbR$ be a locally Lipschitz, definable
   function and $p$ a positive integer. There is $(S_i)$, a definable
   Whitney $C^p$ stratification of $\graph(f)$, such that if one denotes
   by $X_i$ the projection of $S_i$ onto its first $d$
   coordinates, the restriction $f \colon X_i \rightarrow \bbR$ is $C^p$ and the
   family $(X_i)$ is a Whitney $C^p$ stratification of $\bbR^d$.
   Moreover, for any $y \in X_i$ and $v \in \partial f(y)$, we have
   $P_{\cT_y X_i}(v) = \nabla_{X_i} f(y)$.
 \end{lemma}
 Lemma~\ref{lm:f_bol_strat} has important consequences. One of them (see \cite[Section 5]{dav-dru-kak-lee-19})
 is that every locally Lipschitz continuous and definable function is path-differentiable.

 \begin{lemma}[{\cite[Theorem 5.8]{dav-dru-kak-lee-19}}]
   \label{lem:whitney-implest-pathdiff}
   Let $f: \bbR^d \rightarrow \bbR$ be a locally Lipschitz continuous function. If $\graph(f)$ admits a Whitney $C^1$ stratification, then $f$ is path-differentiable.
 \end{lemma}

   \section{Probabilistic Arguments and Proof of Lemma~\ref{duflo-fort}}\label{app:coupl}

  The arguments to prove Lemma~\ref{duflo-fort} are similar to the ones presented in \cite{bra-duf-96, pem-90, tarrespieges}. For future references we have found convenient to present them for a general sequence $(w_n)$, which might not be explicitly related to $(y_n)$.

 Let $d$ be an integer, $(\Omega, \mcF, \bbP)$ be a probability space, $(\mcF_n)$ a filtration on it and $(w_n)$ be a sequence in $\bbR^{d}$  verifying:
 \begin{equation}\label{eq:coupl}
 w_{n+1} = w_n + \gamma_n D_n (w_n) w_n + \gamma_n r_{1,n+1} + \gamma_n r_{2, n+1} + \gamma_n e_{n+1}\, ,
 \end{equation}
 where $(\gamma_n)$ is a real-valued sequence, $(r_{1,n}), (r_{2,n}), (e_{n})$ are $\bbR^d$-valued and adapted to $(\mcF_n)$ and for each $n \in \bbN$, $D_n : \bbR^d \rightarrow \bbR^{d \times d}$ is some measurable map. Furthermore, assume that $w_0$ is $\mcF_0$-measurable.

 Fix a deterministic, real-valued sequence $(\chi_{n})$, four measurable functions $G_1, G_2, G_3, G_4: \bbR^d \rightarrow \bbR_{+}$, an event $\Gamma \in \mcF$ and a measurable function $D: \bbR^d \rightarrow \bbR^d$.
 Assume that we want to prove the fact that $\bbP( \Gamma \cap [w_n \rightarrow 0]) = 0$ under the following set of assumptions.

 \begin{assumption}\label{hyp:hard_coupl}
   On the event $\Gamma \cap [w_n \rightarrow 0]$, the following holds.
   \begin{enumerate}[i)]
     \item\label{hard_coupl_zermean}For all $n \in \bbN$, $\bbE_n e_{n+1} = 0$ and
     \begin{equation*}
       \limsup \bbE_n G_1(e_{n+1}) < + \infty \, , \quad   \liminf \bbE_n G_2(e_{n+1}) > 0\, .
     \end{equation*}
     \item\label{hard_coupl_r1}
     \begin{equation*}
       \sum_{i=0}^{\infty}\norm{r_{1,i+1}}^2 < + \infty  \, .
     \end{equation*}
     \item\label{hard_coupl_r2}
     \begin{equation*}
       \limsup \chi_n \sum_{i=n}^{\infty} \gamma_i \bbE_n\norm{r_{2,i+1}} = 0 \,
     \end{equation*}
     and
     \begin{equation*}
       \limsup \bbE_n G_3(r_{2,n+1}) < +\infty \, , \quad  \lim \bbE_n G_4(r_{2,n+1}) = 0 \, .
     \end{equation*}
         \item\label{hard_coupl_D} $ \limsup \norm{D_n(y_n) - D(y_n)} = 0$.
   \end{enumerate}
 \end{assumption}

 Then, to prove that $\bbP([w_n \rightarrow 0] \cap \Gamma) = 0$ it is sufficient to prove it under the following, more easy to handle, assumption.

 \begin{assumption}\label{hyp:easy_coupl}
  There are six, strictly positive, constants $C_{G_1}, C_{G_2}, C_{G_3}, c_{G_4}, c_{r_1}, c_{r_2}, c_D$, where $c_{\times}$ can be chosen as small as needed, such that almost surely the following holds.
 \begin{enumerate}[i)]
   \item\label{easy_coupl_zermean} For all $n \in \bbN$, $\bbE_n e_{n+1} = 0$ and
   \begin{equation*}
     \sup_{n \in \bbN}\bbE_n G_1(e_{n+1}) \leq C_{G_1} \, , \quad   \inf_{n \in \bbN}\bbE_n G_2(e_{n+1}) \geq C_{G_2}\, .
   \end{equation*}
   \item\label{easy_coupl_r1}
   \begin{equation*}
     \sum_{i=0}^{\infty}\norm{r_{1,i+1}}^2 \leq c_{r_1} \, .
   \end{equation*}
   \item\label{easy_coupl_r2}
   \begin{equation*}
   \sup_{n \in \bbN} \chi_{n}  \sum_{i=n}^{\infty} \gamma_i \bbE_n \norm{r_{2,i+1}} \leq c_{r_2} \,
   \end{equation*}
   and

     \begin{equation*}
       \sup_{n \in \bbN} \bbE_n G_3(r_{2,n+1})\leq C_{G_3} \, , \quad  \sup_{n \in \bbN} \bbE_n G_4(r_{2,n+1})] \leq c_{G_4} \, .
     \end{equation*}
     \item\label{easy_coupl_D} For all $n \in \bbN$, $\norm{D_n(y_n) - D(y_n)} \leq c_D$.
 \end{enumerate}
 \end{assumption}

 To show that proving $\bbP([w_n \rightarrow 0]) = 0$ under Assumption~\ref{hyp:easy_coupl} is sufficient we, for an arbitrary large event $A \in \mcF$, construct a sequence $(\tilde{w}_n)$ that is equal to $(w_n)$ on $\Gamma \cap [w_n \rightarrow 0]$ but satisfies Assumption~\ref{hyp:easy_coupl} almost surely.

 \noindent\textit{\textbf{1)} Equivalence between Assumption~\ref{hyp:hard_coupl}-\eqref{hard_coupl_r1} and Assumption~\ref{hyp:easy_coupl}-\eqref{easy_coupl_r1}.} Notice that
 \begin{equation*}
   \left[\sum_{i=0}^{\infty} \norm{r_{1, i+1}^2} < + \infty\right] \subset \bigcap_{C \in \bbQ} \bigcup_{n_0 \in \bbN}\left[\sum_{i=n_0}^{\infty}\norm{r_{1, i+1}}^2 \leq C\right] \, .
 \end{equation*}
 Therefore, for any $C \in \bbQ$ and $\delta >0$, there is $n_0 \in \bbN$ such that
 \begin{equation*}
   \bbP( \Gamma \cap [w_n \rightarrow 0]) \leq \bbP\left(\Gamma \cap [w_n \rightarrow 0] \cap \left[\sum_{i=n_0}^{\infty}\norm{r_{1, i+1}}^2 \leq C\right]\right) + \delta \, .
 \end{equation*}
 For $n \geq n_0$, denote $(\tilde{r}_{1,n+1})$ the sequence defined as $\tilde{r}_{1,n+1} = r_{1,n+1} \1_{\sum_{i=n_0}^{n+1} \norm{r_{1, i+1}}^2 \leq C}$ and $(\tilde{w}_n)$ the sequence defined as:
 \begin{equation*}
   \tilde{w}_{n+1} =  \tilde{w}_n + \gamma_n D_n (\tilde{w}_n) \tilde{w}_n + \gamma_n \tilde{r}_{1,n+1} + \gamma_n r_{2, n+1} + \gamma_n e_{n+1}\, .
 \end{equation*}
 Shifting the sequences $(\tilde{w}_n), (\gamma_n), (D_n), (\tilde{r}_{n}), (r_n), (e_n)$ by $n_0$, we obtain that the shifted sequences satisfy Assumption~\ref{hyp:hard_coupl}, with Assumption~\ref{hyp:hard_coupl}-\eqref{hard_coupl_r1} replaced by Assumption~\ref{hyp:easy_coupl}-\eqref{easy_coupl_r1}. Furthermore, by construction if we prove that $\bbP([\tilde{w}_n \rightarrow 0]) = 0$, then:
 \begin{equation*}
   \begin{split}
       \bbP(\Gamma \cap [w_n \rightarrow 0 ])&\leq \bbP\left( \Gamma \cap [w_n \rightarrow 0] \cap \left[\sum_{i=n_0}^{\infty}\norm{r_{1, i+1}}^2 \leq C\right]\right) + \delta \\
       &= \bbP\left(\Gamma \cap [\tilde{w}_n \rightarrow 0] \cap \left[\sum_{i=n_0}^{\infty}\norm{r_{1, i+1}}^2 \leq C \right] \right) + \delta\\
       &\leq \bbP([\tilde{w}_n \rightarrow 0]) + \delta  = \delta \, .
   \end{split}
 \end{equation*}
 Since $\delta$ is arbitrary, this will show that $\bbP([w_n \rightarrow 0]) = 0$.

 \noindent\textit{\textbf{2)} Equivalence between Assumption~\ref{hyp:hard_coupl}-\eqref{hard_coupl_r2} and Assumption~\ref{hyp:easy_coupl}-\eqref{easy_coupl_r2}.}
 In a similar manner, for any $\delta, c_{r_2}, c_{G_4} >0$, there is $n_0 \in \bbN$ and $C_{G_3} >0$ such that, up to an event of a probability less than $\delta$, Assumption~\ref{hyp:easy_coupl}-\eqref{easy_coupl_r2} holds for $n \geq n_0$.
 Thus, for $n \geq n_0$, denote $A_n \in \mcF_n$ the event
 \begin{equation*}
   A_n = \left[\sup_{ n_0\leq k \leq n}\chi_k \sum_{i=k}^{\infty} \gamma_i \bbE_k\norm{r_{2,i+1}} \leq c_{r_2}\, , \bbE_n G_3(r_{2,n+1}) \leq C_{G_3}\, , \bbE_n G_4(r_{2,n+1}) \leq c_{G_4}\right] \, .
 \end{equation*}
 Then, define $(\tilde{w}_n)$ with the same recursion than $(w_n)$, but where $r_{2,{n+1}}$ is replaced by $r_{2, n+1} \1_{A_n}$. By construction  Assumption~\ref{hyp:easy_coupl}-\eqref{easy_coupl_r2} is satisfied and proving $\bbP([\tilde{w}_n \rightarrow 0]) = 0$ will imply $\bbP(\Gamma \cap [w_n \rightarrow 0]) \leq \delta$, which, since $\delta$ is arbitrary will show that $\bbP([w_n \rightarrow 0]) = 0$.

 \noindent\textit{\textbf{3)} Equivalence betweeen Assumption~\ref{hyp:hard_coupl}-\eqref{hard_coupl_D} and Assumption~\ref{hyp:easy_coupl}-\eqref{easy_coupl_D}.} As previously, for any $\delta, c_D >0$, there is $n_0 \in \bbN$
 such that for $n \geq n_0$, Assumption \ref{hyp:easy_coupl}-\eqref{easy_coupl_D} is satisfied on $\Gamma \cap [w_n \rightarrow 0]$ up to an event of probability less than $\delta$. Thus, defining, for $n \geq n_0$, $(\tilde{w}_n)$ with the same equation as $(w_n)$, but replacing $D_n(y_n)$ by $D_n(y_n) \1_{\norm{D_n(y_n) - D(y_n)} \leq c_D} + D(y_n) \1_{\norm{D_n(y_n) - D(y_n)}>c_D}$,
  and shifting all of the sequences by $n_0$, we obtain that $(\tilde{w}_n)$ satisfy Assumption~\ref{hyp:easy_coupl}-\eqref{easy_coupl_D} and that
 \begin{equation*}
   \bbP([w_n \rightarrow 0] \cap \Gamma) \leq \bbP([\tilde{w}_n \rightarrow 0]) + \delta \, .
 \end{equation*}

 \noindent\textit{\textbf{4)} Equivalence between Assumption~\ref{hyp:hard_coupl}-\eqref{hard_coupl_zermean} and Assumption~\ref{hyp:easy_coupl}-\eqref{easy_coupl_zermean}.} Here, since we want to preserve the zero-mean assumption on $(e_n)$, the construction is slightly different and goes back to \cite{lai_wei}. Let $(e'_n)$ be a sequence of bounded, zero-mean and i.i.d. random variables defined on an auxiliary probability space $(\Omega', \mcF', \bbP')$. On the probability space $(\Omega \times \Omega', \mcF \otimes \mcF', \bbP \otimes \bbP')$ define a filtration $(\tilde{\mcF}_n)$ as $\tilde{\mcF}_n = \mcF_n \otimes \sigma(\{e'_j : j \leq n \}) $, where $\sigma$ denotes the smallest sigma-algebra generated by a sequence of randoms variables.

 As previously, for any $\delta >0$, there is some $n_0 \in \bbN$ and some constants $C_{G_1}, C_{G_2}$ such that, for $n \geq n_0$, Assumption~\ref{hyp:easy_coupl}-\eqref{easy_coupl_zermean} is satisfied on $\Gamma \cap [w_n \rightarrow 0]$ up to an event with probability less than $\delta$. For $n \geq n_0$, define $A_n \in \mcF_n$ as:
 \begin{equation*}
   A_n = [\bbE_n e_{n+1} = 0, \bbE_n G_1(e_{n+1}) \leq C_{G_1}, \,  \bbE_nG_2(e_{n+1}) \geq C_{G_2}] \, .
 \end{equation*}
 Then, define a sequence of $(\tilde{\mcF}_n)$ adapted random variables $(\tilde{e}_{n})$ as:
 \begin{equation*}
   \forall (a_1, a_2) \in \Omega \times \Omega', \, \quad \tilde{e}_{n}(a_1, a_2) = e_n(a_1)\1_{A_n}(a_1) + e'_n(a_2) \1_{A_n^c}(a_1) \, .
 \end{equation*}
 Let us define $(\tilde{w}_n)$ in the same manner as $(w_n)$ but replacing $(e_n)$ by $(\tilde{e}_n)$ and as previously shift all of the sequences by $n_0$.
 By construction such a sequence satisfies Assumption~\ref{hyp:easy_coupl}-\eqref{easy_coupl_zermean} and as previously:
 \begin{equation*}
   \begin{split}
       \bbP(\Gamma \cap [w_n \rightarrow 0]) &\leq \bbP([\tilde{w}_n \rightarrow 0]) + \delta \, .
   \end{split}
 \end{equation*}
 Thus, proving that $\bbP([\tilde{w}_n \rightarrow 0]) = 0$ will be sufficient to conclude that $\bbP(\Gamma \cap [w_n \rightarrow 0]) = 0$.


%% file: ben_APT.tex
\section{Issues with Asymptotic Pseudotrajectories}\label{app:ben}
In the first version of this paper, Theorem~\ref{th:avoid_trap} was stated in the case where $\cM$ satisfied the following, weaker version, of angle condition. 

\textbf{Angle condition (weak version).} For every $\alpha > 0$, there is $\beta >0$ and a neighborhood $U$ of $x^*$ such that for every $x \in U$, 
\begin{equation}\label{eq:ang_weak}
    f(x) - f(P_{\cM} (x)) \geq \alpha \norm{x - P_{\cM}(x)} \implies \scalarp{v}{x -P_{\cM}(x)} \geq \beta \norm{x - P_{\cM}(x)} \, .
\end{equation}

The advantage of this version over the one used in this work is that Equation~\eqref{eq:ang_weak} is verified by a much larger family of functions than weakly convex ones. Typically, the function $(y,z)\mapsto - y^2 - |z|$ satisfies~\eqref{eq:ang_weak} but not the equation of Definition~\ref{def:ver_ang}. 

The main idea behind the initial (incorrect) proof was to establish the existence of $\alpha >0$, such that, on the event $[x_n \rightarrow x_*]$, for $n$ large enough, the iterates must satisfy the left-hand side of~\eqref{eq:ang_weak}, almost surely. The (weak version) of angle condition then implied that they satisfy the right-hand side of~\eqref{eq:ang_weak}, which is exactly what we need for the proof of Theorem~\ref{th:avoid_trap} (more precisely, to establish Lemmas~\ref{lm:En-z}--\ref{lm:zn2_as_bound} and Proposition~\ref{zsympa}).

Unfortunately, the proof of the first claim was using \cite[Theorem 4.1]{ben-hof-sor-05}, which as noted by an anonymous reviewer was incorrect. To avoid future errors of a similar kind we provide in this section a detailed description of our mistake.

Before going further, let us emphasize, that  up to our knowledge, most of the results based on the work of Bena\"im, Hofbauer and Sorin are using \cite[Theorem 4.2]{ben-hof-sor-05}, which, appropriately restated (see Proposition~\ref{ben:true}), remains valid. In particular, we are not aware of any other works, where a mistake is due to an incorrect use of \cite[Theorem 4.1]{ben-hof-sor-05}.

\paragraph{Differential inclusions.} 
In the following, we fix a set-valued map $\sH: \bbR^d \rightrightarrows \bbR^d$ (i.e. $\forall x\in \bbR^d$, $\sH(x)\subset\bbR^d$) that satisfies assumptions of \cite{ben-hof-sor-05}. The relevant example for our work is the case where $\sH = - \partial f$.
We
say that an absolutely continuous curve $\sx:\bbR_{+} \rightrightarrows \bbR^d$ is a solution of the differential inclusion (DI)
\begin{equation}
  \dot{\sx}(t) \in \sH(\sx(t))\, ,
\label{eq:dif_incl}
\end{equation}
with initial condition $x_0\in \bbR^d$, if
$\sx(0) = x_0$ and \eqref{eq:dif_incl} holds for almost every  $t\in\bbR_{+}$. We denote $S_{\sH}(x_0)$ the set of these solutions.

\paragraph{Stochastic approximation of differential inclusions.} Let $(\gamma_n)$ be a sequence of decreasing step-sizes and $(x_n)$ a recursive algorithm satisfying 
\begin{equation}\label{eq:stoch_approx}
  x_{n+1} \in x_n + \gamma_n \sH(x_n) + \gamma_n U_{n+1} \, ,
\end{equation}
where $(U_n)$ is an $\bbR^{d}$-valued sequence of perturbations. One of the purposes of the work of Bena\"im, H\"ofbauer and Sorin (\cite{ben-hof-sor-05}) was to show that under mild conditions on 
$(\gamma_n)$ and $(U_n)$ the iterates produced by Equation~\eqref{eq:stoch_approx} closely follow a solution of the DI~\eqref{eq:dif_incl}.

Indeed, let us define the linearly interpolated process $\sX: \bbR_{+} \rightarrow \bbR^d$ as: 
\begin{equation}\label{eq:lin_int_pro}
  \sX(t) = x_n + \frac{t- \sum_{i=0}^{n} \gamma_i}{\gamma_{n+1}} (x_{n+1} - x_n)\, , \quad \textrm{ if } \sum_{i=0}^n \gamma_i\leq t \leq \sum_{i=0}^{n+1} \gamma_i\, .
\end{equation}
Then \cite[Proof of Theorem~4.2]{ben-hof-sor-05} shows the following result.

\begin{proposition}\label{ben:true}
  Assume the following. \emph{i)} The iterates $(x_n)$ are bounded. \emph{ii)} The step-sizes satisfy $\gamma_n \rightarrow 0$ and $\sum_{i=0}^{\infty} \gamma_i = + \infty$. \emph{iii)} For all $T >0$ it holds that 
  \begin{equation*}
     \sup_{n \leq k \leq \tau(n, T)} \norm{\sum_{i =n}^{k-1} \gamma_i U_{i+1}} \xrightarrow[n \rightarrow \infty]{}0 \, , \quad \textrm{  where $ \tau(n,T) = \inf \left\{ k \geq n : \sum_{i=n}^k \gamma_i  \geq T\right\}$}\, .
  \end{equation*}
Then, the family $(\sX(t+ \cdot))_{t \in \bbR_{+}}$ is relatively compact (in the space of bounded, continuous functions with the norm of uniform convergence on compact sets) and for any $t_n \rightarrow + \infty$ and $\sz : \bbR_{+} \rightarrow \bbR^d$ such that $\sX(t_n + \cdot) \rightarrow \sz$, it holds that $\sz$ is a solution to~\eqref{eq:dif_incl}.
\end{proposition}
\paragraph{Asymptotic pseudotrajectory.} The asymptotic result of the previous paragraph is quite strong and is actually sufficient for most of the results of \cite{ben-hof-sor-05} about accumulation points of iterates given by a stochastic approximation recursion (in particular they are ICT sets). However, the linearly interpolated process~\eqref{eq:lin_int_pro} was claimed to possess a stronger property.

Following \cite{ben-hof-sor-05}, a continuous curve $\sX: \bbR_{+} \rightarrow \bbR^d$ is said to be an asymptotic pseudotrajectory (APT) of~\eqref{eq:dif_incl} if for every $T >0$, it holds that 
\begin{equation}\label{eq:APT}
    \lim_{t \rightarrow + \infty}\sup_{h\in [0,T]} \inf_{ \sx \in S_{\sH}(\sX(t))}\norm{X(t+h) - \sx(h)} = 0 \, .
\end{equation}

The following result was claimed to be true. 
\begin{proposition}\label{prop:ben_false}\cite[Theorem 4.1]{ben-hof-sor-05} 
  Assume that $\sz: \bbR_{+} \rightarrow \bbR^d$ is bounded. Then there is equivalence between. 
  \begin{enumerate}
    \item\label{equiv_b1} $\sz$ is an APT.
    \item\label{equiv_b2} $\sz$ is uniformly continuous and if there is $t_n \rightarrow + \infty$ and $\tilde{\sz} : \bbR_{+} \rightarrow \bbR^d$ such that $\sz(t_n + \cdot) \rightarrow \tilde{\sz}$ (in the sense of uniform convergence on compact intervals), then $\tilde{\sz}$ is a solution to the DI~\eqref{eq:dif_incl}.
  \end{enumerate}
\end{proposition}
In the first version of this work, the implication \ref{equiv_b2} $\implies$ \ref{equiv_b1} was used to obtain the fact that asymptotically the iterates of SGD satisfy the left-hand side of~\eqref{eq:ang_weak} (and thus, by the weak angle condition, the right-hand side). Following that, the proof was identical to the one of Theorem~\ref{th:avoid_trap}. Unfortunately, while \ref{equiv_b1} $\implies$ \ref{equiv_b2}, we actually do not have the reverse implication. Indeed, the following example was provided by an anonymous reviewer. 

Consider $\sz(t) = \frac{1}{t+1}$ and $\sH  = \partial | \cdot |$. Thus, $\sH(x) = \sign(x)$ if $x \neq 0$ and $\sH(0) = [-1, 1]$. The curve $\sz$ is uniformly continuous, and for any $T>0$ and $t_n \rightarrow + \infty$, 
\begin{equation*}
  \sup_{h \in [0,T]}\norm{\sz(t_n + h)} \rightarrow 0 \, .
\end{equation*}
Since $\tilde{\sz}\equiv 0$ is a solution to the DI associated to $\sH$ we indeed have \ref{equiv_b2}. Nevertheless, for any $n \in \bbN$ and $\sx$ a solution to the DI~\eqref{eq:dif_incl} starting at $\sz(t_n)$, we have $\sx(T) = \sz(t_n) +T$. Thus, 
\begin{equation}
  \lim_{t_n \rightarrow + \infty}\sup_{h\in [0,T]} \inf_{ \sx \in S_{\sH}(\sz(t_n))}\norm{\sz(t_n+h) - \sx(h)} \geq T\, ,
\end{equation}
and the first point of Proposition~\ref{prop:ben_false} is not satisfied.

\end{appendices}